\documentclass[letterpaper,11pt]{amsart}
\usepackage{amsmath,amsthm,amssymb,mathtools,amsfonts,mathrsfs}
\usepackage{xspace,xcolor}
\usepackage[breaklinks,colorlinks,citecolor=teal,linkcolor=teal,urlcolor=teal,pagebackref,hyperindex]{hyperref}
\usepackage[alphabetic]{amsrefs}
\usepackage{tikz-cd}
\usepackage[all]{xy}
\usepackage{bbm}
\usepackage{MnSymbol}
\usepackage{stmaryrd}

\newtheorem{theorem}{Theorem}[section]
\newtheorem{prop}[theorem]{Proposition}
\newtheorem{proposition}[theorem]{Proposition}

\newtheorem{thm}[theorem]{Theorem}
\newtheorem{lemma}[theorem]{Lemma}

\newtheorem{assumption}[theorem]{Assumption}

\theoremstyle{definition} 
\newtheorem{defn}[theorem]{Definition}
\newtheorem{definition}[theorem]{Definition}
\newtheorem{example}[theorem]{Example}

\newtheorem{remark}[theorem]{Remark}

\newtheorem{question}[theorem]{Question}

\newcommand{\cO}{\mathcal{O}}

\newcommand{\qu}{/\kern-.7ex/}
\newcommand{\lqu}{\backslash \kern-.7ex \backslash}
\newcommand{\on}{\operatorname}

\newcommand{\Hom}{\on{Hom}}

\newcommand{\bM}{\overline{\mathcal M}}
\newcommand{\bt}{\mathbf{t}}

\title[Relative QH under birational transformations]{Relative quantum cohomology under birational transformations}

\author{Fenglong You}
\address{Department of Mathematics, University of Oslo, PO Box 1053 Blindern,
0316 Oslo, Norway}
\email{youf@math.uio.no}

\thanks{}


\begin{document}
\date{\today}

\begin{abstract} 
We study how relative quantum cohomology, defined in \cite{TY20c} and \cite{FWY}, varies under birational transformations. For toric complete intersections with simple normal crossings divisors that contain the loci of indeterminacy, we prove that their respective relative $I$-functions can be directly identified. For toric complete intersections with smooth divisors, we prove that their respective relative $I$-functions are related by analytic continuation. We also study some connections with extremal transitions and FJRW theory.
\end{abstract}

\maketitle 

\tableofcontents

\section{Introduction}

\subsection{Motivations}\label{sec:motivation}

Let $X$ be a smooth projective variety and $D\subset X$ be a divisor which is either smooth or simple normal crossings. In joint work with H. Fan and L. Wu \cite{FWY}, we defined relative Gromov--Witten theory with (possibly) negative contact orders which generalized relative Gromov--Witten theory of \cite{Li1}, \cite{Li2}, \cite{LR} and \cite{IP}. Our generalization allows us to define quantum cohomology for relative Gromov--Witten theory, namely relative quantum cohomology. The result of \cite{FWY} has been generalized to simple normal crossing pairs in joint work with H.-H. Tseng in \cite{TY20c}. The invariants for pairs in \cite{FWY} and \cite{TY20c} are limits of orbifold Gromov--Witten invariants of (multi-)root stacks when roots are taken to be sufficiently large. So we formally consider them as Gromov--Witten invariants of infinite root stacks. Since these invariants are naturally defined for the pairs, the quantum cohomology constructed in \cite{FWY} and \cite{TY20c} are called relative quantum cohomology.

The main purpose of this paper is to study how relative quantum cohomology varies under birational transformations. There are at least three motivations for studying this question. 

The first motivation is to understand the relation between the formal Gromov--Witten theory of infinite root stacks in \cite{TY20c} and the (punctured) logarithmic Gromov--Witten theory of \cite{ACGS}, \cite{AC} and \cite{GS13}. It is well-known that log Gromov--Witten invariants are birational invariance \cite{AW}. However, the formal Gromov--Witten invariants of infinite root stacks defined in \cite{TY20c} are not birational invariance in general as pointed out by Dhruv Ranganathan. Therefore, in principle, understanding how the formal Gromov--Witten invariants of infinite root stacks vary under birational transformations should tell us how these two theories differ. We predict that structures of the formal Gromov--Witten theory of infinite root stacks are still invariant under birational transformations although single invariants are not.

The second motivation is to generalize the crepant/discrepant transformation conjecture which is a central topic in absolute Gromov--Witten theory, see for example \cite{Ruan}, \cite{BG09}, \cite{CCIT09}, \cite{CY}, \cite{GW}, \cite{LLW10}, \cite{ILLW}, \cite{LLW16a}, \cite{LLW16b}, \cite{LLW16c}, \cite{CIJ}, \cite{AS18}, \cite{AS}, \cite{Iritani20}. The crepant transformation conjecture predicts that, for K-equivalent smooth varieties or Deligne--Mumford stacks, their quantum cohomology should be related by analytic continuation. We propose a generalization and a reformation of this conjecture by associating the targets with appropriate divisors such that the pairs are ``log-K-equivalent".  Following the spirit of Ruan's crepant transformation conjecture, it should be reasonable to expect that any Gromov--Witten theory associated to pairs should be ``invariant" under log-crepant birational transformations. 

The third motivation is to use the degeneration scheme \cite{MP} to prove the crepant transformation conjecture in general. The degeneration technique plays an important role in the crepant transformation conjecture. For example, it has already been shown in \cite{LLW10} that the invariance of quantum rings under ordinary flops can be reduced to the case of relative local models by applying the degeneration formula. One can try to apply mirror symmetry and consider the mirrors of degenerations. The Doran--Harder--Thompson conjecture \cite{DHT} states that if a Calabi--Yau manifold admits a Tyurin degeneration to two quasi-Fano varieties, then the Landau--Ginzburg models of these two quasi-Fanos can be glued together to the mirror of the original Calabi--Yau manifold with a $\mathbb P^1$-fibration. In joint work with C. Doran and J. Kostiuk \cite{DKY}, we proved that a B-model gluing formula for toric complete intersections that relates relative $I$-functions of two quasi-Fanos (with their smooth anticanonical divisors) and the absolute $I$-function of the Calabi--Yau. The Doran--Harder--Thompson conjecture and the gluing formula in \cite{DKY} have recently been generalized to the degeneration of quasi-Fano varieties \cite{DKY21}. The gluing formula is expected to be true in general. The general gluing formula, the crepant transformation conjecture for pairs and the degeneration scheme  together may be used to prove the genus zero crepant transformation conjecture as long as the varieties/Deligne-Mumford stacks admit good degenerations. Therefore, although we are mostly focusing on the discrepant case, the crepant case is also important.

We expect that by replacing absolute quantum cohomology by relative quantum cohomology, we are able to have a unified and simplified discussion of crepant and discrepant transformations. We expect that relative quantum cohomology behaves better under birational transformations. For example, The natural map from the root stack $X_{D,r}$ to $X$ is birational. It was proved in \cite{TY16} and \cite{CCR} that the Gromov--Witten theory of $X_{D,r}$ is determined by the Gromov--Witten theory of $X$, Gromov--Witten theory of $D$ and the restriction map $H^*(X)\rightarrow H^*(D)$. On the other hand, the relation between relative Gromov--Witten theory of $(X_{D,r},D_r)$ and $(X,D)$ is significantly simpler: they coincide by \cite{AF}.

\subsection{Set-up}\label{sec:intro-set-up}

We consider a birational transformation $\phi: X_+\dashedrightarrow X_-$ between smooth varieties or Deligne--Mumford stacks. Suppose there is a smooth variety or a Deligne--Mumford stack $\tilde X$ with projective birational  morphisms $f_{\pm}:\tilde{X}\rightarrow X_{\pm}$ such that the following diagram commute  
\[
\begin{tikzcd}
 & \arrow[dl,swap]{}{f_+} \tilde X \arrow[dr]{}{f_-} \\
X_+  \arrow[rr,dashed]{}{\phi} && X_-.
\end{tikzcd}
\]
We consider divisors $D_+\subset X_+$ and $D_-\subset X_-$ such that
\[
f_+^*(K_{X_+}+D_+)=f_-^*(K_{X_-}+D_-).
\]

In this paper, we focus on the genus zero Gromov--Witten theory. Instead of comparing single invariants, we would like to take a structural approach by considering relative quantum cohomology. When comparing relative quantum cohomology, we need to decide which divisors that we choose. Different choices of the divisors will lead to very different relations between their relative quantum cohomology. There are at least two natural choice of divisors:
\begin{enumerate}
\item Simple normal crossings divisors containing the loci of indeterminacy such that the pairs are ``log-K-equivalent".
\item Smooth divisors such that the pairs are ``log-K-equivalent".
\end{enumerate}

We consider Case (1) as the most natural case. Heuristically, Gromov--Witten theory of a pair $(X,D)$ may be considered as an enumerative theory for the complement $X\setminus D$ (at least in some special cases). If $X_+$ and $X_-$ are birational and we choose $D_+\subset X_+$ and $D_-\subset X_-$ such that $D_+$ and $D_-$ contain the loci of indeterminacy, and $X_+\setminus D_+$ and $X_-\setminus D_-$ are isomorphic, then we should expect Gromov--Witten theories of $(X_+,D_+)$ and $(X_-,D_-)$ are equivalent in some sense. 

The relation in Case (2) is expected to be more complicated. But it will have lots of connections to classical results in Gromov--Witten theory. The invariants that we consider in this case are simply relative Gromov--Witten invariants of \cite{LR}, \cite{IP},  \cite{Li1}, \cite{Li2}, as well as the generalization of relative Gromov--Witten invariants of smooth pairs with negative contact orders in \cite{FWY} and \cite{FWY19}. In this case, we also have an additional motivation from mirror symmetry. In terms of mirror symmetry, it is very natural to consider Gromov--Witten invariants of (smooth or simple normal crossings) pairs. The Fano/Landau--Ginzburg mirror symmetry states that the mirror to a Fano variety $X$ is a Landau--Ginzburg model $(X^\vee,W)$ which is a variety $X^\vee$ with a superpotential $W: X^\vee \rightarrow \mathbb C$. Then one expects that the generic fiber of the superpotential $W$ is mirror to the smooth anticanonical divisor $D\subset X$. Therefore, one indeed considers the Landau--Ginzburg model $(X^\vee, W)$ as the mirror to the pair $(X,D)$. Since many proofs for crepant transformation conjecture are based on mirror symmetry, It is very natural to consider relative Gromov--Witten theory of smooth pairs at least in the Fano case. 

\subsection{Results}

The main results are in the setting of toric wall-crossings. In other words, we consider toric Deligne--Mumford stacks $X_\pm$ of the form $[\mathbb C^m\sslash_{w_\pm}K ]$ that are related by a single wall-crossing in the space of stability conditions.

We investigate the following cases:
\begin{enumerate}
    \item GIT wall-crossing for toric complete intersection stacks with simple normal crossings toric divisors containing the loci of indeterminacy.
    \item GIT wall-crossing for toric complete intersection stacks with smooth nef divisors.
\end{enumerate}

The main ingredients for our proof are the relative mirror theorems in \cite{FTY} and \cite{TY20b}, where the authors constructed relative $I$-functions and showed that relative $I$-functions lie in Givental's Lagrangian cone for the pairs. In this paper, it is more convenient to consider a function called the $H$-function which is related to the $I$-function via the $\hat{\Gamma}$-class (see Equation (\ref{H-I})).

\subsubsection{Case (1)}

\begin{theorem}[Theorem \ref{thm-toric-stack-snc} and Theorem \ref{thm-toric-ci}]\label{thm-intro-toric-snc}
Suppose a birational transformation $\phi: X_+\dashedrightarrow X_-$ between toric Deligne--Mumford stacks (or toric complete intersections) is given by a single toric wall-crossing. Let $D_+$ and $D_-$ be simple normal-crossings divisors such that  toric divisors containing the loci of indeterminacy  are the irreducible components, then their relative $H$-functions are directly identified (without analytic continuation).
\end{theorem}

\begin{remark}
This birational invariance of the formal orbifold Gromov--Witten theory of infinite root stacks may be considered as both expected and unexpected. It is unexpected because single orbifold invariants are not birational invariants as pointed out by D. Ranganathan. On the other hand, it is expected because the basic principle of the crepant transformation conjecture suggests that the formal orbifold Gromov--Witten theory of infinite root stacks should be invariant under log-crepant transformations on the level of generating functions.
\end{remark}

We also study a variant of Case (1) for toric complete intersections in Section \ref{sec:exchange-divisors}. Given a convex line bundle $L\rightarrow X$, one can consider the Gromov--Witten theory of the zero locus $D$ of a regular section of $L$. On the other hand, one can also consider the relative Gromov--Witten theory of $(X,D)$. The mirrors of $D$ and $(X,D)$ are closely related. For simplicity, we explain it in the case of a Fano variety $X$ with its smooth anticanonical divisor $D$. As the motivation that we mentioned in Section \ref{sec:intro-set-up}, the mirror Landau--Ginzburg model of $(X,D)$ can be considered as a family of Calabi--Yau manifolds that are mirror to $D$. On the other hand, the mirror of $D$ is also a family of Calabi--Yau manifolds of the same type. These two families of Calabi--Yau manifolds are closely related. Therefore, one may expect that the Gromov--Witten theory of $(X,D)$ and $D$ are closely related. When $D$ is not anticanonical, this motivation still makes sense by considering higher rank Landau--Ginzburg models in \cite{DKY21}. We will also explain it in Section \ref{sec:exchange-divisors}.

Suppose we have line bundles $L_{1,\pm}$ and $L_{2,\pm}$ over $X_\pm$ satisfying the assumption in Section \ref{sec:toric-ci}. Let $s_{i,\pm}$ be regular sections of $L_{i,\pm}\rightarrow X_\pm$ for $i=1,2$. Let $Z_+\subset X_+$ be the hypersurface defined by $s_{1,+}$ and $Z_-\subset X_-$ be the hypersurface defined by $s_{2,-}$. Let $D_{2,+}$ be the smooth divisor of $X_+$ defined by $s_{2,+}$ and $D_{1,-}$ be the smooth divisor of $X_-$ defined by $s_{1,-}$.

Let $D_+$ and $D_-$ be the simple normal crossings divisors chosen in Theorem \ref{thm-intro-toric-snc}. Then we consider 
\[
D_{Z,+}=(D_{2,+}+D_+)\cap Z_+
\]
and 
\[
D_{Z,-}=(D_{1,-}+D_-)\cap Z_-.
\]
Then we can compare the ambient parts of the relative quantum cohomology of $(Z_+,D_{Z,+})$ and $(Z_-,D_{Z,-})$ pulled back from the corresponding relative quantum cohomology of the ambient toric varieties $X_\pm$.

\begin{theorem}[Theorem \ref{thm-exchange-divisor}]
The relative $H$-functions of $(Z_+,D_{Z,+})$ and $(Z_-,D_{Z,-})$ can be identified directly (without analytic continuation).
\end{theorem}

 Theorem \ref{thm-exchange-divisor} provides lots of examples that were not considered in previous works. For example, one can consider Gromov--Witten theories of $\mathbb P^3$ and a smooth quartic threefold $Q_4$ relative to their smooth anticanonical $K3$ surfaces respectively. Note that $\mathbb P^3$ and the quartic threefold are both hypersurfaces of the same ambient toric variety $\mathbb P^4$. There is no wall-crossing in this example: $X_+=X_-=\mathbb P^4$, $L_1=\mathcal O_{\mathbb P^4}(1)$ and $L_2=\mathcal O_{\mathbb P^4}(4)$. The mirrors to $(\mathbb P^3, K3)$ and $(Q_4,K3)$ are families of $M_2$-polarized $K3$ surfaces. It is not difficult to see that the classical periods for their Landau--Ginzburg mirrors (Laurent polynomials) are related by a change of variables. See for example \cite{CCGK} for a list of classical periods for Fano threefolds. See also Example \ref{ex-exchange-divisor}.

\subsubsection{Case (2)}

We first consider relative Gromov--Witten invariants with maximal contact orders. By the local-relative correspondence of \cite{vGGR} (see also \cite{TY20b}), these invariants of the smooth pair $(X,D)$ coincide (up to some constant factors) with local Gromov--Witten invariants of $\mathcal O_{X}(-D)$. Therefore, we just need to compare these local invariants. Then it becomes similar to the approach in \cite{MS} where the authors study Gromov--Witten theory under extremal transitions using quantum Serre duality. The varieties related by transitions in \cite{MS} are hypersurfaces in toric varieties where the ambient toric varieties are related by toric blow-ups.

We focus on the type of toric birational transformations called discrepant resolutions. With a mild generalization of \cite{MS} to discrepant resolutions, we prove the following statement in terms of quantum D-module.

\begin{theorem}[=Theorem \ref{thm-narrow-qd}]\label{thm-intro-narrow-qd}
After analytic continuation, the monodromy invariant part of the narrow quantum D-module of $(X_+,D_+)$ around $y^\prime=0$, when restricted to $y^\prime=0$, contains  the narrow quantum D-module of $(X_-,D_-)$ as a subquotient.
\end{theorem}

\begin{remark}
One can also obtain a result when the birational transformation is crepant. Then the result will follow from the crepant transformation conjecture of \cite{CIJ} and the $I$-functions are simply related by analytic continuations.
\end{remark}

In order to use the local-relative correspondence, we focus on the relative invariants with one relative marking of maximal contact order. In general, we would like to remove the restriction of maximal contact orders. Therefore, we consider the orbifold Gromov--Witten invariants of root stack $X_{D,r}$. In our set-up of Case (2), the root stack $X_{D,r}$ is not a toric stack in general. However, by the hypersurface construction of \cite{FTY}, the root stack $X_{D,r}$ is a hypersurface of $Y_{X_{\infty},r}$, where $Y=\mathbb P_X(\mathcal O(-D)\oplus \mathcal O)$ is a toric stack and $Y_{X_{\infty},r}$ is the $r$-th root of $Y$ along the infinity divisor $X_{\infty}$. Then the toric stacks $Y_+$ and $Y_-$ are related by two wall-crossings. The first wall-crossing is crepant and the second wall-crossing is discrepant. By taking the limit $r\rightarrow \infty$ and applying analytic continuation to relative $I$-functions we have the following. 

\begin{theorem}[=Theorem \ref{thm-sm-extended}]
If $\phi: \xymatrix{
X_+\ar@{-->}[r]& X_-}$ is a discrepant resolution induced from a toric wall-crossing, the (extended) $H$-function of $(X_+,D_+)$ can be analytically continued to the $H$-function of $(X_-,D_-)$ under the specialization of a variable $y^\prime=0$. 
\end{theorem}

\begin{remark}
One can consider a combination of Case (1) and Case (2) by allowing the irreducible components of the divisors to be either toric divisors or nef (not necessarily toric invariant) divisors. Then we will also have the analytic continuation between their relative $I$-functions.
\end{remark}

\subsection{Connection to other works}

\subsubsection{Logarithmic Gromov--Witten theory}
It is well-known that log Gromov--Witten invariants are birational invariance \cite{AW}.  Instead of comparing single invariants as in \cite{AW}, we take a structural approach by studying the genus zero generating functions ($J$-functions). This version of birational invariance for orbifold Gromov--Witten theory may suggest a possible relation between log and orbifold invariants on the level of generating functions.

\subsubsection{Transition}\label{sec:intro-trans}

In \cite{MS}, the authors compared genus zero Gromov--Witten theory of toric hypersurfaces related by extremal transitions arising from toric blow-ups. The quantum D-modules are related by analytic continuations and specializations of parameters. The rank reduction phenomenon was partially studied in \cite{Mi} for an example of cubic extremal transition in terms of FJRW theory. The general case has not been studied yet. We provide a partial explanation for the rank reduction phenomenon in general directly in terms of (relative/local) Gromov--Witten theory of the subvariety that is contracted under the transition.  

Let $\phi: \xymatrix{\tilde X\ar@{-->}[r]& X}$ be a toric blow-up. Let $\tilde D\subset \tilde X$ and $D\subset X$ be the divisors $D_+$ and $D_-$ chosen in Case (2). Then $\tilde D$ and $D$ are related by transitions. The main result of \cite{MS} (see also \cite{Mi}) can be stated in terms of $I$-functions as follows. There is an explicit degree-preserving linear transformation $L: H^*(\tilde D)\rightarrow H^*(D)$ such that
\[
I_{ D}(y)=\lim_{\tilde y_{r+1}\rightarrow 0} L\circ \bar I_{\tilde D}(\tilde y),
\]
where $\bar I_{\tilde D}(\tilde y)$ is obtained from $I_{\tilde D}(\tilde y)$ via analytic continuation. 

Let $S$ be the variety that gets contracted under the transition $\xymatrix{
\tilde D\ar@{-->}[r]& D}$. Then the local Gromov--Witten invariants of the normal bundle $N_{S/\tilde X}$ should account for the rank reduction phenomenon. We have the following.

\begin{theorem}[=Theorem \ref{thm-transition}]\label{thm-intro-transition}
The restriction of the $I$-function of $\tilde D$ to $S$ coincides with the $I$-function $I_{N_{S/\tilde X}}$. Under this identification, we have 
$I_{N_{S/\tilde X}}$
vanishes under the analytic continuation and specialization in \cite{MS}.
\end{theorem}


\begin{remark}
Theorem \ref{thm-intro-transition} also gives an explanation for the rank reduction phenomenon in Theorem \ref{thm-intro-narrow-qd}.
\end{remark}

\subsubsection{FJRW theory}

The Gromov--Witten theory of $N_{S/\tilde X}$ considered in Section \ref{sec:intro-trans} can be regarded as Gromov--Witten theory of Fano hypersurfaces relative to its smooth anticanonical divisor. On the other hand, the rank reduction phenomenon was interpreted in terms of FJRW theory in \cite{Mi} for an example of cubic extremal transition. It is natural to ask if this relative Gromov--Witten theory is directly related to the FJRW theory. 

The celebrated Landau--Ginzburg/Calabi--Yau (LG/CY) correspondence states that the Gromov--Witten theory of a Calabi--Yau variety can be identified with the FJRW theory of a singularity via analytic continuation. In genus zero, it has been proved in various cases in \cite{CR}, \cite{CIR}, \cite{LPS}, \cite{CR18}, \cite{Zhao}. The LG/CY correspondence has been generalized to the case of Fano or general type varieties \cite{Acosta} and \cite{AS}, where the authors proved the relation between Gromov--Witten theory of a Fano variety or a general type variety and the FJRW theory of a singularity. Such a relation is known as the LG/Fano (general type) correspondence. 

The proofs of the above genus zero LG/CY correspondence and the LG/Fano (general type) correspondence are based on mirror symmetry and the relation between their $I$-functions. Recall that one can understand mirror symmetry for Fano varieties as mirror symmetry for log Calabi--Yau pairs. In other words, the mirror of a Fano variety $X$ with its smooth anticanonical divisor $-K_X$ is a Landau--Ginzburg model. Therefore, it is natural to consider the \emph{LG/(log CY) correspondence} instead of the LG/Fano correspondence. By studying the relative $I$-function of $(X,-K_X)$ and the regularized $I$-function of the FJRW theory, we have the following.

\begin{theorem}\label{thm-intro-log-lg-cy}[=Theorem \ref{thm-log-lg-cy}]
The narrow relative quantum D-module of $(X,-K_X)$ can be identified with the narrow quantum D-module of $(W,G)$ via analytic continuation. 
\end{theorem}

\subsection{Future directions}

Although our main results are based on toric wall-crossings, we expect similar results in more general setting. Our results also lead to many interesting questions to be explored in the future. We only mention some of them here.

\subsubsection{Birational invariance and log Gromov--Witten invariants}
We expect that the birational invariance of relative quantum cohomology that we study in this paper is true in general. Whenever one can construct the relevant $I$-functions, One can expect similar relations to hold.

The $I$-functions are related to the $J$-functions via mirror maps or generalized mirror maps \cite{Iritani07}. On the other hand, the (inverse) mirror map is expected to be the instanton corrections (see, for example, \cite{CLT}, \cite{CCLT} and \cite{CLLT}). It may be reasonable to expect that the orbifold Gromov--Witten invariants of $X_{D,\infty}$ together with the instanton corrections are birational invariance. The instanton corrections are open Gromov--Witten invariants. The open/closed duality predicts that these open Gromov--Witten invariants are equal to some relative/log Gromov--Witten invariants. 

It is also interesting to see if one can find a precise relation between the orbifold and log Gromov--Witten invariants through the property of birational invariance. This idea is also investigated in a recent work of Battistella--Nabijou--Ranganathan \cite{BNR22} via a different method.

\subsubsection{Regularized $I$-function}

In \cite{AS}, the authors studied discrepant toric birational transformations using regularized $I$-functions and asymptotic expansions. It might be interesting to see how the approach in \cite{AS} is related to our approach. It might also be interesting to study the relation with \cite{Iritani20}.

\subsubsection{The degeneration scheme and the crepant transformation conjecture}

As a general principle of the degeneration scheme, it should be possible to understand the crepant transformation conjecture in general by understanding the relative version of the crepant transformation conjecture of simpler targets. This also motivates the idea of understanding the degeneration scheme on the structural level.

\subsection*{Acknowledgements}
F.Y. is supported by the Research Council of Norway grant no. 202277. We would like to thank Qile Chen, Dhruv Ranganathan, Jørgen Vold Rennemo, Mark Shoemaker and Hsian-Hua Tseng for helpful discussions.

\section{Relative quantum cohomology}
In this section, we consider Gromov--Witten theory of a smooth projective variety $X$ relative to a simple normal crossings divisor $D=D_1+\ldots+D_n$. Instead of considering logarithmic Gromov--Witten theory of \cite{AC}, \cite{Chen},\cite{GS13} and \cite{ACGS}, we will consider the the formal Gromov--Witten theory of the infinite root stack $X_{D,\infty}$ defined in \cite{TY20c}. When the divisor $D$ is smooth (that is, set $n=1$), we recover the relative Gromov--Witten theory of \cite{FWY} and \cite{FWY19}.

The formal Gromov--Witten theory of infinite root stacks $X_{D,\infty}$ is defined as a ``limit" of orbifold Gromov--Witten theory of multi-root stacks $X_{D,\vec r}$. To simplify the presentation, we only write down the definition when $X$ is a smooth projective variety, but the definition in \cite{TY20c} also applies to smooth proper Deligne--Mumford stacks as mentioned in \cite{TY20c}. 

\subsection{Orbifold Gromov--Witten theory}

Given a smooth proper Deligne-Mumford stack $\mathcal X$ with projective coarse moduli space $X$. One can consider the moduli space $\bM_{0,l}(\mathcal X, \beta)$ of $l$-pointed, genus $g$, degree $\beta\in H_2(X)$ stable maps to $\mathcal X$.

The orbifold Gromov--Witten invariants of $\mathcal X$ are defined as follows
\begin{align}
\left\langle \prod_{i=1}^l \tau_{a_i}(\gamma_i)\right\rangle_{g,l,\beta}^{\mathcal X}:=\int_{[\bM_{g,l}(\mathcal X, \beta)]^{\on{vir}}}\prod_{i=1}^l(\on{ev}^*_i\gamma_i)\bar{\psi}_i^{a_i},
\end{align}
where, 
\begin{itemize}
\item $[\bM_{g,l}(\mathcal X, \beta)]^{\on{vir}}$ is the the virtual fundamental class of $\bM_{g,l}(\mathcal X, \beta)$. 
\item 
for $i=1,2,\ldots,l$,
\[
\on{ev}_i: \bM_{g,l}(\mathcal X,\beta) \rightarrow I\mathcal X
\]
is the evaluation map and $I\mathcal X$ is the inertia stack of $\mathcal X$.
\item  $\gamma_i\in H_{\on{CR}}^*(\mathcal X)$, where $H_{\on{CR}}^*(\mathcal X)$ is the Chen-Ruan orbifold cohomology of $\mathcal X$.
\item $a_i\in \mathbb Z_{\geq 0}$, for $1\leq i\leq l$.
\item
$\bar{\psi}_i\in H^2(\bM_{g,l}(\mathcal X, \beta),\mathbb Q)$
is the descendant class.
\end{itemize}

\subsection{Gromov--Witten theory of root stacks}
Let $r_1, ..., r_n\in \mathbb{N}$ be pairwise coprime, the multi-root stack $$X_{D, \vec{r}}:=X_{(D_1, r_1),...,(D_n, r_n)}$$ is a smooth Deligne--Mumford stack.

For any index set $I\subseteq\{1,\ldots, n\}$, we define 
\[
D_I:=\cap_{i\in I} D_i.
\]
Let 
\[
\vec s=(s_1,\ldots,s_n)\in \mathbb Z^n.
\] 
We define
\[
I_{\vec s}:=\{i:s_i\neq 0\}\subseteq \{1,\ldots,n\}.
\]

Consider the vectors
\[
\vec s^j=(s_1^j,\ldots,s_n^j)\in \mathbb Z^n, \text{ for } j=1,2,\ldots,m,
\]
which satisfy the following condition:
\[
\sum_{j=1}^m s_i^j=\int_\beta[D_i], \text{ for } i\in\{1,\ldots,n\}.
\]
For sufficiently large\footnote{By sufficiently large $\vec r$, we mean $r_i$ are sufficiently large for all $i\in\{1,\ldots,n\}$.} $\vec r$, we consider the moduli space $$\bM_{g,\{\vec s^j\}_{j=1}^m}(X_{D,\vec r},\beta)$$ of genus $g$, degree $\beta\in H_2(X)$, $m$-pointed, orbifold stable maps to $X_{D,\vec r}$ with orbifold conditions specified by $\{\vec s^j\}_{j=1}^m$. 

Let \begin{itemize}
    \item $\gamma_j\in H^*(D_{I_{\vec s^j}})$, for $j\in\{1,2,\ldots,m\}$;
    \item $a_j\in \mathbb Z_{\geq 0}$, for $j\in \{1,2,\ldots,m\}$.
\end{itemize}
Gromov-Witten invariants of $X_{D,\vec r}$ are defined as follows
\begin{align*}
    \left\langle \gamma_1\bar{\psi}^{a_1},\ldots, \gamma_m\bar{\psi}^{a_m} \right\rangle_{g,\{\vec s^j\}_{j=1}^m,\beta}^{X_{D,\vec r}}:=
    \int_{\left[\bM_{g,\{\vec s^j\}_{j=1}^m}(X_{D,\vec r},\beta)\right]^{\on{vir}}}\on{ev}_1^*(\gamma_1)\bar{\psi}_1^{a_1}\cdots\on{ev}_m^*(\gamma_m)\bar{\psi}_m^{a_m}.
\end{align*}
We define
\[
s_{i,-}:=\#\{j: s_i^j<0\}, \text{ for } i=1,2,\ldots, n.
\]
By \cite{TY20c}*{Corollary 16},
\[
\left(\prod_{i=1}^n r_i^{s_{i,-}}\right)\left\langle \gamma_1\bar{\psi}^{a_1},\ldots, \gamma_m\bar{\psi}^{a_m} \right\rangle_{g,\{\vec s^j\}_{j=1}^m,\beta}^{X_{D,\vec r}}
\]
is a polynomial in $r_1,\ldots, r_n$ when $\vec r$ is sufficiently large. It is constant in $\vec r$ when $g=0$ and $\vec r$ is sufficiently large.

The formal Gromov--Witten invariants of $X_{D,\infty}$ can be defined as follows.
\begin{defn}\label{defn:GW_inv}
Let \begin{itemize}
    \item $\gamma_j\in H^*(D_{I_{\vec s^j}})$, for $j\in\{1,2,\ldots,m\}$;
    \item $a_j\in \mathbb Z_{\geq 0}$, for $j\in \{1,2,\ldots,m\}$.
\end{itemize}
The formal Gromov--Witten invariants of $X_{D,\infty}$ are defined as
\begin{align*}
    \left\langle [\gamma_1]_{\vec s^1}\bar{\psi}^{a_1},\ldots, [\gamma_m]_{\vec s^m}\bar{\psi}^{a_m} \right\rangle_{g,\{\vec s^j\}_{j=1}^m,\beta}^{X_{D,\infty}}:=\left[\left(\prod_{i=1}^n r_i^{s_{i,-}}\right)\left\langle \gamma_1\bar{\psi}^{a_1},\ldots, \gamma_m\bar{\psi}^{a_m} \right\rangle_{g,\{\vec s^j\}_{j=1}^m,\beta}^{X_{D,\vec r}}\right]_{\prod_{i=1}^n r_i^0}
\end{align*}
for sufficiently large $\vec r$, where $[f(\vec r)]_{\prod_{i=1}^n r_i^0}$ of a polynomial $f(\vec r)$ in $r_1,\ldots,r_n$ means taking the constant term of the polynomial. When $f(\vec r)$ is a constant in $r_1,\ldots, r_n$, we have $[f(\vec r)]_{\prod_{i=1}^n r_i^0}=f(\vec r)$.  
\end{defn}
\begin{remark}
The formal Gromov--Witten theory of the infinite root stack $X_{D,\infty}$ was defined for the case when $X$ is a smooth projective variety. However, as mentioned in \cite{TY20c}, the generalization to the case when $X$ is a smooth Deligne--Mumford stack is straightforward. We just need to replace the cohomology rings in the state space by the relevant orbifold cohomology rings.   
\end{remark}

\begin{remark}
When $D$ has only one irreducible component, the formal Gromov--Witten theory of $X_{D,\infty}$ is simply the relative Gromov--Witten theory in \cite{FWY} and \cite{FWY19} which is a generalization of the classical relative Gromov--Witten theory of  \cite{LR}, \cite{IP},  \cite{Li1}, \cite{Li2} to allow negative contact orders. 
\end{remark}

\subsection{Relative quantum cohomology}

Following \cite{TY20c}, the state space of relative quantum cohomology of $X_{D,\infty}$ is defined as
\[
\mathfrak H:=\bigoplus_{\vec s\in \mathbb Z^n}\mathfrak H_{\vec s},
\]
where 
\[
\mathfrak H_{\vec s}:=H^*(D_{I_{\vec s}}).
\]

The pairing on $\mathfrak H$ 
\[
(-,-):\mathfrak H \times \mathfrak H\rightarrow \mathbb C
\]
is defined as
\begin{equation}\label{eqn:pairing}
\begin{split}
([\alpha]_{\vec s},[\beta]_{\vec s^\prime}) = 
\begin{cases}
\int_{D_{I_{\vec s}}} \alpha\cup\beta, &\text{if } \vec s=-\vec s^\prime;\\
0, &\text{otherwise. }
\end{cases}
\end{split}
\end{equation}
The pairing on the rest of the classes is generated by linearity.

Given a basis $\{T_{I,k}\}_k$ for $H^*(D_I)$, we can define a basis of $\mathfrak H$ as follows:
\[
\tilde{T}_{\vec s,k}=[T_{I_{\vec s},k}]_{\vec s}.
\]
Let $\{T_{I}^k\}$ be the dual basis of $\{T_{I,k}\}$ under the Poincar\'e pairing of $H^*( D_I)$. We define a dual basis of $\{\tilde T_{\vec s, k}\}$ under the pairing of $\mathfrak H$ as follows:
\[
\tilde{T}_{\vec s}^k=[T_{I_{\vec s}}^k]_{\vec s}.
\]
Let $ t=\sum t_{\vec s,k}\tilde{T}_{\vec s,k}$ where $t_{\vec s,k}$ are formal variables. The genus-zero potential for the Gromov-Witten theory of infinitely root stacks is defined to be
\[
\Phi_0( t)=\sum_{m\geq 3}\sum_{\beta}\frac{1}{m!}\langle t,\cdots,t\rangle_{0,m,\beta}^{X_{D,\infty}} q^{\beta}.
\]
The relative quantum product $\star$ is defined as
\[
\tilde{T}_{\vec s^1,k_1}\star \tilde{T}_{\vec s^2,k_2}=\sum_{\vec s^3,k_3}\frac{\partial^3 \Phi_0}{\partial t_{\vec s^1,k_1}\partial t_{\vec s^2,k_2}\partial t_{\vec s^3,k_3}}\tilde{T}_{-\vec s^3}^{k_3}.
\]

\subsection{Givental formalism and mirror theorem}
Consider the space
\[
\mathcal H=\mathfrak H \otimes_{\mathbb C}\mathbb C[\![\on{NE}(X)]\!](\!(z^{-1})\!),
\]
where $(\!(z^{-1})\!)$ means formal Laurent series in $z^{-1}$.

There is a $\mathbb C[\![\on{NE}(X)]\!]$-valued symplectic form
\[
\Omega(f,g)=\text{Res}_{z=0}(f(-z),g(z))dz, \text{ for } f,g\in \mathcal H,
\]
where the pairing $(f(-z),g(z))$ takes values in $ \mathbb C[\![\on{NE}(X)]\!](\!(z^{-1})\!)$ and is induced by the pairing on $\mathfrak H$.

We consider the following polarization
\[
\mathcal H=\mathcal H_+\oplus\mathcal H_-,
\]
where
\[
\mathcal H_+=\mathfrak H \otimes_{\mathbb C} \mathbb C[\![\on{NE}(X)]\!][z], \quad \text{and} \quad \mathcal H_-=z^{-1}\mathfrak H \otimes_{\mathbb C} \mathbb C[\![\on{NE}(X)]\!][\![z^{-1}]\!].
\]
There is a natural symplectic identification between $\mathcal H_+\oplus \mathcal H_-$ and the cotangent bundle $T^*\mathcal H_+$.

For $l\geq 0$, we write $ t_l=\sum\limits_{\vec s,k} t_{l;\vec s,k}\widetilde T_{\vec s,k}$ where $t_{l;\vec s,k}$ are formal variables. We write
\[
\mathbf t(z)=\sum\limits_{l=0}^\infty t_l z^l.
\]
The genus zero descendant Gromov-Witten potential of $X_{D,\infty}$ is defined as
\[
\mathcal F^0_{X_{D,\infty}}(\bt(z))=\sum\limits_\beta \sum\limits_{m=0}^\infty \dfrac{q^\beta}{m!} \left\langle\mathbf t(\bar\psi),\ldots,\mathbf t(\bar\psi)\right\rangle_{0,m,\beta}^{X_{D,\infty}}.
\]

Givental's Lagrangian cone $\mathcal L_{X_{D,\infty}}\subset \mathcal H=T^*\mathcal H_+$ is defined as the graph of the differential $d\mathcal F^0_{X_{D,\infty}}$.
In other words, a (formal) point in the Lagrangian cone can be explicitly written as
\[
-z+\mathbf t(z)+\sum\limits_{\beta} \sum\limits_{m} \sum\limits_{\vec s,k} \dfrac{q^\beta}{m!} \left\langle\dfrac{\tilde T_{\vec s,k}}{-z-\bar\psi},\mathbf t(\bar\psi),\ldots,\mathbf t(\bar\psi)\right\rangle_{0,m+1,\beta}^{X_{D,\infty}} \tilde T_{-\vec s}^k.
\]

\begin{defn}
We define the $J$-function $J_{X_{D,\infty}}(t,z)$ as follows,
\[
J_{X_{D,\infty}}(t,z)=z+t+\sum_{m\geq 1, \beta\in \on{NE}(X)}\sum_{\vec s,k}\frac{q^\beta}{m!}\left\langle\dfrac{\tilde T_{\vec s,k}}{-z-\bar\psi},t,\ldots, t\right\rangle_{0,m+1,\beta}^{X_{D,\infty}} \tilde T_{-\vec s}^k.
\]
\end{defn}

The $I$-function $I_{X_{D,\infty}}$ for $X_{D,\infty}$ is constructed in \cite{TY20b}*{Section 4} as a hypergeometric modification of the $J$-function of $X$:
\begin{align}\label{I-snc}
I_{X_{D,\infty}}(Q,t,z):=\sum_{\beta\in \on{NE}(X)} J_{X, \beta}(t,z)Q^{\beta}
\prod_{i=1}^n\prod_{0< a< d_i}(D_i+az)[\textbf{1}]_{ (-d_1,-d_2,\ldots, -d_n)}.
\end{align}
A mirror theorem for the infinite root stack $X_{D,\infty}$ can be stated as follows using Givental formalism.
\begin{thm}[\cite{TY20c}, Theorem 29]\label{thm:mirror}
Let $X$ be a smooth projective variety. Let $D:=D_1+D_2+...+D_n$ be a simple normal-crossing divisor with $D_i\subset X$ smooth, irreducible and nef. The $I$-function $I_{X_{D,\infty}}$, defined in \cite{TY20b}*{Section 4}, of the infinite root stack $X_{D,\infty}$ lies in Givental's Lagrangian cone $\mathcal L_{X_{D,\infty}}$ of $X_{D,\infty}$.
\end{thm}

Recall that $ t=\sum t_{\vec s,k}\tilde{T}_{\vec s,k}$. If we set $ t_{\vec s,k}=0$ for $\vec s\neq \vec 0$, invariants in the $J$-function only have one relative marking. Hence, this relative marking must have maximal contact order. When $D_i$ are nef, these invariants are identified with local invariants of $\bigoplus_{i=1}^n\mathcal O_X(-D_i)$:

\begin{theorem}[\cite{BNTY}, Theorem 1.1]\label{thm-local-orbifold}
Let $\beta$ be a curve class of $X$ with $d_i:=D_i\cdot \beta>0$ for $i\in \{1,\ldots, n\}$. Consider $m$ marked points $x_1,\ldots,x_m$ and fix an ordered partition of the index set $\{1,\ldots,n\}$ into disjoint subsets $I_1,\ldots,I_m$ such that $\cap_{i \in I_j} D_i$ is nonempty for each $j \in \{1,\ldots,m\}$. The following identity is true:
\begin{align}
   &\left[\left(\cup_{j=1}^m \on{ev}_j^*(\cup_{i\in I_j}D_i)\right)\cap\bM_{0,m}(\oplus_{i=1}^n\mathcal O_X(-D_i),\beta)\right]^{\on{vir}}\\
   \notag =&\left(\prod_{i=1}^n(-1)^{d_i-1}\right)F_*[\bM_{0,0,\{(d_i)\}_{i\in I_1},\ldots, \{(d_i)\}_{i \in I_m}}(X_{D, \infty},\beta)]^{\on{vir}},
\end{align}
where 
$$F:\bM_{g,l,(d)}(X_{D,\infty},\beta)\to \bM_{g,l}(X,\beta)$$ obtained by forgetting the orbifold structure. 
\end{theorem}

In particular, their $I$-functions were already identified in \cite{TY20b}*{Section 5.2}.

\subsection{Relative quantum D-module}

The Dubrovin connection for relative quantum cohomology is given by a collection of operators
\[
\nabla_{\vec s,k}=\nabla_{\frac{\partial}{\partial t_{\vec s,k}}}=\frac{\partial}{\partial t_{\vec s,k}}+\frac{1}{z}\tilde{T}_{\vec s,k}\star_t 
\]
and
\[
\nabla_{z\partial z}=z\frac{\partial }{\partial z}-\frac{1}{z}E\star_t+\mu,
\]
where  $E$ is the Euler vector field
\begin{align}
E:=c_1(T_X(-\log D))+\sum_{\vec s,k} \left(1-\frac 12 \deg(\tilde{T}_{\vec s,k})-\#\{i:s_i<0\}\right)t_{\vec s,k}\tilde{T}_{\vec s,k}
\end{align}
and the grading operator
\begin{align}\label{grading-operator}
\mu(\tilde{T}_{\vec s,k})=\left(\frac 12 \deg(\tilde{T}_{\vec s,k})+\#\{i:s_i<0\}-\frac 12 \dim_{\mathbb C}X\right)\tilde{T}_{\vec s,k}.
\end{align}
This meromorphic connection is flat by standard arguments in \cite{CK} and the WDVV equation in \cite{TY20c}.

The equation $\nabla f=0$ is called the quantum differential equation. We define 
\begin{align}
    L(t,z)\alpha:=\alpha+\sum_{d\in \on{NE}(X)_{\mathbb Z}, \vec s, k}Q^d\sum_{l\geq 0}\frac{1}{l!}\left\langle \frac{\alpha}{-z-\psi},t,\ldots,t,\tilde{T}_{\vec s,k}\right\rangle\tilde T_{-\vec s}^k,
\end{align}
where $\on{NE}(X)$ is the cone of effective curve classes and $\on{NE}(X)_{\mathbb Z}:=\on{NE}(X)\cap H_2(X,\mathbb Z)$.

\begin{proposition}
The operator $L(t,z)$ satisfies the following differential equations:
\begin{align}\label{qde-1}
    \nabla_{\vec s,k}L(t,z)\alpha=0,
\end{align}
and
\begin{align}\label{qde-2}
    \nabla_{z\partial z}L(t,z)\alpha=L(t,z)(\mu\alpha-\frac{\rho}{z}\alpha),
\end{align}
where $\alpha\in \mathfrak H$, $\rho:=c_1(T_X(-\log D))\in H^2(X)$ and $\mu$ is the grading operator (\ref{grading-operator}).
\end{proposition}

\begin{proof}
Equation (\ref{qde-1}) follows from the topological recursion relation for the Gromov--Witten theory of $X_{D,\infty}$ (\cite{TY20c}*{Proposition 24}).

The proof of Equation (\ref{qde-2}) is also identical to the corresponding equation for orbifold Gromov--Witten theory (see \cite{Iritani09}*{Proposition 2.4}). It follows from the homogeneity of Gromov--Witten invariants and Equation (\ref{qde-1}). 

\end{proof}
The operator $L(t,z)$  is called a fundamental solution to the quantum differential equation.

\begin{proposition}
Set 
\[
z^{-\mu}z^{\rho}:=\exp(-\mu\log z)\exp(\rho \log z)
\]
Then we have
\begin{align}\label{identity-L}
    \nabla_{\vec s,k}\left(L(t,z)z^{-\mu}z^{\rho}\alpha\right)=0, \quad  \nabla_{z\partial z}\left(L(t,z)z^{-\mu}z^{\rho}\alpha\right)=0
\end{align}
and
\begin{align}\label{pair-L}
    (L(t,-z)\alpha, L(t,z)\beta)=(\alpha,\beta),
\end{align}
where $\alpha,\beta\in \mathfrak H$ and $(-,-)$ is the pairing (\ref{eqn:pairing}).
\end{proposition}
\begin{proof}
It follows from the proof of \cite{Iritani09}*{Proposition 2.4}. Equation (\ref{identity-L}) follows from (\ref{qde-1}), (\ref{qde-2}) and the differential equation
\[
(z\partial_z+\mu-\rho/z)(z^{-\mu}z^\rho\alpha)=0.
\]
It remains to show Equation (\ref{pair-L}). By (\ref{qde-1}) and the Frobenius property of the relative quantum product $\star$, we have
\begin{align*}
&\frac{\partial}{\partial t_{\vec s,k}}(L(t,-z)\alpha, L(t,z)\beta)\\
=&\frac{1}{z}(\tilde{T}_{\vec s,k}\star_t L(t,-z)\alpha, L(t,z)\beta)-\frac{1}{z}(L(t,-z)\alpha, \tilde{T}_{\vec s,k}\star_t L(t,z)\beta)\\
=&0
\end{align*}
Therefore, the pairing $(L(t,-z)\alpha, L(t,z)\beta)$ is constant in $t_{\vec s, k}$. The initial condition implies Equation (\ref{pair-L}).
\end{proof}


\begin{remark}
  The cone $\mathcal L_{X,D}$ is reconstructed as
  \[
  \mathcal L_{X,D}=\bigcup_{t}zL(t,-z)^{-1}\mathcal H_+.  
  \]
  The $J$-function is
  \[
  J(t,z)=L(t,z)^{-1}[\textbf{1}]_{\vec 0}.
  \]

\end{remark}

If $D_i$ is nef, for $1\leq i\leq n$, and $t$ is restricted to $\mathfrak H_0$, then, by Theorem \ref{thm-local-orbifold}, the restricted relative quantum $D$-module is identified with the narrow quantum $D$-module of $\bigoplus_{i=1}^n\mathcal O_X(-D_i)$ defined in \cite{Shoemaker18}. We will also refer to this part of the relative quantum $D$-module as the \emph{narrow relative quantum $D$-module} of $(X,D)$.


\section{Set-up}\label{sec:set-up}

We consider a birational transformation $\phi: X_+\dashedrightarrow X_-$ between smooth varieties/Deligne--Mumford stacks. Suppose there is a smooth variety/Deligne--Mumford stack $\tilde X$ with projective birational  morphisms $f_{\pm}:\tilde{X}\rightarrow X_{\pm}$ such that the following diagram commute  
\[
\begin{tikzcd}
 & \arrow[dl,swap]{}{f_+} \tilde X \arrow[dr]{}{f_-} \\
X_+  \arrow[rr,dashed]{}{\phi} && X_-.
\end{tikzcd}
\]
Then $X_+$ and $X_-$ are called $K$-equivalent if 
\[
f_+^*(K_{X_+})=f_-^*(K_{X_-}).
\]
The celebrated crepant transformation conjecture of Y. Ruan predicts that the quantum cohomology of $X_+$ and $X_-$ should be related by analytic continuation. 

The case when $X_+$ and $X_-$ are not $K$-equivalent is more subtle. There have been some works investigating how Gromov--Witten theory varies under discrepant transformations, see, for example \cite{Iritani20}, \cite{AS18} and \cite{AS}. The relation is usually more complicated than the crepant case. For example, in \cite{AS}, the authors use regularization and asymptotic expansion to study how genus zero Gromov--Witten theory varies under discrepant toric wall-crossing. We propose a different approach to this question by introducing divisors $D_+$ and $D_-$ associated to $X_+$ and $X_-$ respectively such that it becomes ``log-K-equivalent". In other words, we consider (smooth or simple normal crossings) divisors $D_+\subset X_+$ and $D_-\subset X_-$ such that
\begin{align}\label{identity-log-crepant}
f_+^*(K_{X_+}+D_+)=f_-^*(K_{X_-}+D_-).
\end{align}
We would like to claim that considering Gromov--Witten theories of the pairs can give us simpler relations and provide new insights and connections to previous works.

In the spirit of the crepant transformation conjecture, it is natural to ask: 
\begin{question}
What is the relation between the Gromov--Witten theory of $(X_+,D_+)$ and the Gromov--Witten theory of $(X_-,D_-)$?
\end{question} 

Up to this point, we have not yet specified which divisors that we would like to consider. Different choices of divisors can lead to different relations between their Gromov--Witten theories with different complexity. There are several natural choices of the divisors. 
\begin{itemize}
    \item The first natural choice is to consider simple normal crossing divisors 
\[
D_+=D{+,1}+\cdots+ D_{+,n_+}\subset X_+ \text{ and } D_-=D_{-,1}+\cdots+D_{-,n_-}\subset X_-
\]
which contain the loci of indeterminacy of $\phi$ and $\phi^{-1}$ such that Identity (\ref{identity-log-crepant}) holds.
\item Another natural choice is to find smooth divisors $D_+$ and $D_-$ such that Identity (\ref{identity-log-crepant}) holds. 
\end{itemize}

These are two very different choices, but we will see that they both lead to very interesting phenomena.  

In this paper, we focus on the genus zero invariants and study the relation between their relative quantum cohomology. The higher genus case is beyond the scope of this paper and will be studied in the future.

\begin{example}
Let $X_+$ be a blow-up of $X_-$ along a complete intersection center $D_{-,1} \cap\cdots \cap D_{-,n}$. We can choose
\begin{align}\label{ex-blow-up-D--}
D_-=D_{-,1}+\cdots + D_{-,n},
\end{align}
and
\begin{align}\label{ex-blow-up-D-+}
D_+=D_{+,1}+\cdots +D_{+,n}+E,
\end{align}
where $D_{+,i}$ are strict transforms of $D_{-,i}$ and $E$ is the exceptional divisor.

Instead of requiring the irreducible components of $D_\pm$ to consist of $D_{\pm,i}$ and the exceptional divisor, we can also simply let $D_+$ be another divisor that is linear equivalent to $D_{+,1}+\cdots +D_{+,n}+E$ and $D_-$ be another divisor that is linear equivalent to $D_{-,1}+\cdots + D_{-,n}$. In particular, we may choose $D_\pm$ to be smooth divisors if possible. Then, we can compare the relative Gromov--Witten theories of $(X_+,D_+)$ and $(X_-,D_-)$.  

It is natural to expect that different choices of divisors will lead to different relations between their relative quantum cohomology. If we choose $D_-$ and $D_+$ as in (\ref{ex-blow-up-D--}) and (\ref{ex-blow-up-D-+}), then we can expect a simpler relation between relative quantum cohomology of $(X_+,D_+)$ and $(X_-,D_-)$. If we choose $D_-$ and $D_+$ to be smooth divisors, we may expect the relation to be more complicated.
\end{example}

\section{Toric birational transformations}\label{sec:toric-birational}

\subsection{Toric set-up}

We consider toric Deligne--Mumford stacks given by the following GIT data:
\begin{itemize}
    \item a torus $K\cong (\mathbb C^\times)^{\mathrm{r}}$ of rank $\mathrm{r}$;
    \item the cocharacter lattice $\mathbb L=\Hom(\mathbb C^*,K)$ of $K$;
    \item a set of characters $D_1,\ldots, D_m\in \mathbb L^\vee=\Hom (K,\mathbb C^*)$ of $K$;
    \item a choice of a stability condition $\omega\in \mathbb L^\vee\otimes \mathbb R$.
\end{itemize}

The characters $D_1,\ldots,D_m$ define a map from $K$ to the torus $T=(\mathbb C^\times)^m$, hence induce an action of $K$ on $\mathbb C^m$. 

For $I\subset \{1,\ldots,m\}$, we write $\bar{I}$ for the complement of $I$. We define $\angle_I$ to be the following
\[
\angle_I:=\left\{\sum_{i\in I}a_iD_i| a_i\in \mathbb R_{>0}\right\}\subset \mathbb L^\vee \otimes \mathbb R.
\]
Let
\[
(\mathbb C^\times)^{I}\times (\mathbb C)^{\bar{I}}:=\left\{(z_1,\ldots,z_m)| z_i\neq 0 \text{ for } i\in I\right\}.
\]
For the stability condition $\omega\in \mathbb L^\vee \otimes \mathbb R$, we define the set of anticones
\[
\mathcal A_\omega:=\{I\subset \{1,\ldots,m\}| \omega \in \angle_I\}.
\]
and the open set
\[
\mathcal U_\omega:=\bigcup_{I\in \mathcal A_\omega}(\mathbb C^\times)^{I}\times (\mathbb C)^{\bar{I}}.
\]
The toric stack is defined as
\[
X_\omega:=[\mathcal U_\omega/K],
\]
where the square brackets here indicate that $X_\omega$ is the stack quotient of $\mathcal U_\omega$ by $K$.

We always assume that $\omega$ is in the nonnegative span $\sum_{i=1}^m \mathbb R_{\geq 0}D_i$ which ensures that $X_\omega$ is non-empty. The space of stability conditions has a wall and chamber structure. We define
\[
C_\omega:=\bigcap_{I\in \mathcal A_\omega}\angle_I\subset \mathbb L^\vee \otimes \mathbb R.
\]
Then $\omega\in C_\omega$. For $\omega^\prime\in C_\omega$, we have $X_\omega=X_{\omega^\prime}$. We also assume that $C_\omega$ is of maximal dimension. This ensures that $X_\omega$ is a Deligne--Mumford stack. The walls of $\sum_{i=1}^m \mathbb R_{\geq 0}D_i$ are formed by higher codimension $C_\omega$. The chamber $C_\omega$ is called the extended ample cone.

The GIT data are in one-to-one correspondence with (extended) stacky fans defined by \cite{BCS} (\cite{Jiang} for extended stacky fans).
Recall that an $S$-extended stacky fan consists of a quadruple $(N,\Sigma,\beta,S)$, where 
\begin{itemize}
\item
$N$ is a finitely generated abelian group; 
\item
$\Sigma \subset N_{\mathbb{R}}=N\otimes_{\mathbb{Z}}\mathbb{R}$ 
is a rational simplicial fan;
\item
$\beta:\mathbb{Z}^{m}\rightarrow N$ is a map given by $\{b_1,\cdots, b_{m}\}\subset N$, where $b_i=\beta(e_i)\in N$ are images of the $i$-th standard basis $e_i\in \mathbb Z^m$;
\item $S\subset \{1,\ldots, m\}$. The vectors $b_i$ for $i\in S$ are called extended vectors.
\end{itemize}

We denote the image of $b_i$ in $N\otimes \mathbb R$  as $\bar{b}_i$. For $i\in\{1,\ldots,m\}\setminus S$, $\bar{b}_i$'s are vectors determining the rays of the stacky fan. For $i\in S$, $\bar{b}_i$ lies in the support $|\Sigma|$ of the fan.

For the correspondence with the GIT data, we consider the fan sequence:
\begin{equation}\label{fan-seq}
0 \longrightarrow \mathbb{L}:=\text{ker}(\beta) \longrightarrow \mathbb{Z}^{m} \stackrel{\beta}{\longrightarrow} N\longrightarrow 0,
\end{equation}
where the map $\mathbb L\rightarrow \mathbb Z^m$ is given by $(D_1,\ldots,D_m)$.

For $I\subset\{1,2,\cdots,m\}$, let $\sigma_I$ be the cone generated by 
$\bar{b}_i$, for $i\in I$. Let $\overline{I}$ be the complement of $I$ in $\{1,2,\cdots,m\}$. Then the fan $\Sigma_\omega$ is defined as
\[
\Sigma_\omega:=\{\sigma_I|\overline{I}\in \mathcal A_\omega\}.
\]
The extended data $S$ is given by
\[
S:=\{i\in \{1,\ldots,m\}|\overline{\{i\}}\not\in \mathcal A_\omega\}.
\]
The dual of the fan sequence (\ref{fan-seq}) is
\[
0\longrightarrow N^\vee\longrightarrow (\mathbb Z^m)^\vee\longrightarrow \mathbb L^\vee.
\]
It is called the divisor sequence. It induces the exact sequence
\[
0\longrightarrow N^\vee\otimes \mathbb R\longrightarrow \mathbb R^{m-|S|}\longrightarrow \mathbb L^\vee\otimes \mathbb R/\sum_{i\in S}\mathbb R D_i\longrightarrow 0.
\]
We have
\[
H^2(X_\omega,\mathbb R)\cong \mathbb L^\vee\otimes \mathbb R/\sum_{i\in S}\mathbb R D_i.
\]
Moreover, there is a canonical splitting
\begin{align}\label{splitting}
\mathbb L^\vee\otimes \mathbb R\cong H^2(X_\omega,\mathbb R)\oplus \sum_{i\in S}\mathbb R D_i
\end{align}
as described in \cite{Iritani09}*{Section 3.1.2}. Under the splitting (\ref{splitting}), we also have a splitting of the extended ample cone $C_\omega$:
\[
C_\omega\cong C_\omega^\prime\times \sum_{i\in S}\mathbb R D_i\subset H^2(X_\omega,\mathbb R)\oplus \sum_{i\in S}\mathbb R D_i,
\]
where $C_\omega^\prime$ is the ample cone. The dual cone of $C_\omega^\prime$ is called the Mori cone:
\[
NE(X_\omega)=C_\omega^{\prime,\vee}=\left\{d\in H_2(X_\omega;\mathbb R):\eta\cdot d \geq 0 \text{ for all } \eta\in C_\omega^\prime\right\}.
\]

We define the subset $\mathbb K\subset \mathbb L\otimes \mathbb Q$ as
\[
\mathbb K:=\left\{f\in \mathbb L\otimes \mathbb Q|\left\{i \in \{1,\ldots,m\}| D_i\cdot f\in \mathbb Z \right\}\in \mathcal A\right\}.
\]

By \cite{BCS}, components of the inertia stack $IX_\omega$ are indexed by elements of the set $\on{Box}(X_\omega)$, where
\[
\on{Box}(X_\omega)=\left\{ v\in N | \bar{v}=\sum_{i\not\in I}c_i\bar{b}_i\in N\otimes \mathbb R \text{ for some } I\in \mathcal A \text{ and } 0\leq c_i<1\right\}.
\]
There is an isomorphism 
\begin{align}\label{isom-K-Box}
\mathbb K/\mathbb L \cong \on{Box}(X_\omega)
\end{align}
given by
\[
[f]\in \mathbb K/\mathbb L \mapsto v_f=\sum_{i=1}^m \lceil -(D_i\cdot f)\rceil b_i\in N,
\]
where $\lceil \cdots\rceil$ is the ceiling function. That is, for a rational function $q$, $\lceil q\rceil$ is the smallest integer $n$ such that $q\leq n$. Similarly, we write $\lfloor q\rfloor$ for the floor function, that is, the largest integer $n$ such that $n\leq q$;  and $\langle q\rangle$ for the fractional part, that is $\langle q\rangle=q-\lfloor q\rfloor$. 

The isomorphism (\ref{isom-K-Box}) shows that components of $IX_\omega$ are indexed by elements of $\mathbb K/ \mathbb L$. The component of $IX_\omega$ corresponding to $f\in \mathbb K/\mathbb L$ is denoted by $X_\omega^f$. The age $\iota_f$ of the component $X_\omega^f\subset IX_\omega$ is $\sum_{i=1}^m \langle D_i\cdot f\rangle$. We write $\textbf{1}_{f}\in H^0( X_{\omega}^f)$ for the fundamental class of the twisted sector $X_{\omega}^f$.

\subsection{Wall crossing}\label{sec:wall-crossing}

Let $C_+$ and $C_-$ be chambers in $\mathbb L^\vee \otimes \mathbb R$ that are separated by a hyperplane wall $W$ such that $\overline{C_W}:=W\cap \overline{C_+}=W\cap\overline{C_-}$ is a facet of $\overline{C_+}$ and $\overline{C_-}$. We choose stability conditions $\omega_+\in C_+$ and $\omega_-\in C_-$ and set $X_+:=X_{\omega_+}$ and $X_-:=X_{\omega_-}$. Let $\phi:X_+\dashedrightarrow X_-$ be the birational transformation induced by the toric wall-crossing. Note that the birational transformation $\phi$ is crepant if $\sum_{i=1}^m D_i\in W$. Otherwise, it is discrepant. Let $e$ be a primitive generator of $W^\perp$ such that $\omega_+\cdot e>0$ and $\omega_-\cdot e<0$.

Let $\omega_0$ be in the relative interior of $W\cap \overline{C_+}=W\cap \overline{C_-}$. Consider
\[
\mathcal A_0:=\mathcal A_{\omega_0}=\{I\subset \{1,\ldots,m\}: \omega_0\in \angle_I\}.
\]

Set
\[
M_{\pm}:=\{i\in \{1,\ldots,m\}:\pm D_i\cdot e>0\};
\]
\[
M_0=\{i\in \{1,\ldots,m\}: D_i\cdot e=0\};
\]
\[
\mathcal A_0^{\on{thin}}:=\{I\in \mathcal A_0: I\subset M_0\};
\]
\[
\mathcal A_0^{\on{thick}}:=\{I\subset \mathcal A_0: I \cap M_+\neq 0, I\cap M_-\neq 0\}.
\]

\begin{lemma}[\cite{CIJ}*{Lemma 5.2}]
We have
\[
M_0\in \mathcal A_0^{\on{thin}};
\]
\[
\mathcal A_0=\mathcal A_0^{\on{thin}} \cup \mathcal A_0^{\on{thick}};
\]
\[
\mathcal A_{\pm}=\mathcal A_0^{\on{thick}}\cup \{I\cup J: \empty \neq J\subset M_{\pm},I\in \mathcal A_0^{\on{thin}}\}.
\]
\end{lemma}
Recall that the fan $\Sigma_\omega$ is defined as
\[
\Sigma_\omega=\{\sigma_I: \bar{I}\in \mathcal A_\omega\}.
\]
The fans corresponding to $X_+$ and $X_-$ are given by
\[
\Sigma_\pm=\{\sigma_I: \bar{I}\in \mathcal A_\pm\}.
\]
The $S$-extended data for $X_\pm$ are defined as
\[
S_\pm=\left\{i\in \{1,\ldots,m\}|\overline{\{i\}}\not\in \mathcal A_{\pm}\}\right\}.
\]
Following \cite{CIJ} and \cite{Iritani20}, there are three types of toric wall-crossings:
\begin{enumerate}
    \item[(I)] Flip/flop: $X_+$ and $X_-$ are isomorphic in codimension one.
    \item[(II)] Discrepant/crepant resolution: the morphism $X_+\rightarrow |X_-|$ or $X_-\rightarrow |X_+|$ contracts a divisor to a toric substack.
    \item[(III)] $|X_+|$ and $|X_-|$ are isomorphic, but the stack structures along a divisor are different or the gerbe structures are different.
\end{enumerate}

In terms of stacky fans, we have $S_0:=S_+\cap S_-\subset M_0$. Furthermore
\begin{enumerate}
    \item[(I)] For flip and flop, the rays are the same: $S_+=S_-$, $\# (M_+)\geq 2$ and $\#(M_-)\geq 2$.
    \item[(II)] For discrepant or crepant resolution, one either remove or add one ray. If it is removing one ray, then $S_-=S_+\sqcup \{i\}$, $M_-=\{i\}$ and $\#\{M_+\}\geq 2$. If it is adding one ray, then $S_+=S_-\sqcup \{i\}$, $M_+=\{i\}$ and $\#\{M_-\}\geq 2$. 
    \item[(III)] If $|X_+|$ and $|X_-|$ are isomorphic, then there exists $i_+,i_-\in \{1,\ldots, m\}$ such that $S_+=S_0\sqcup \{i_+\}$, $S_-=S_0\sqcup\{i_-\}$, $M_+=\{i_+\}$ and $M_-=\{i_-\}$.
\end{enumerate}

\begin{remark}
 For type (III), the birational transformations change gerbe structure or root structure. Absolute and relative Gromov--Witten theories under gerbe and root constructions have been studied a lot in the literature (e.g. \cite{AJT15}, \cite{TT}, \cite{TT19}, \cite{AF},\cite{ACW},\cite{TY16},\cite{TY18}, \cite{TY20c}.)
\end{remark}

\begin{remark}\label{rmk-circuit}
Recall that a circuit is a set of minimal linearly dependent vectors. Following \cite{BH} and \cite{GKZ} (see also \cite{CIJ}*{Remark 5.3}), a toric wall-crossing can also be described as a modification along the circuit $\{b_i: i\in M_+\cup M_-\}$. To turn the fan $\Sigma_+$ to $\Sigma_-$, One removes every cone $\sigma_I$ of $\Sigma_+$ such that $I$ contains $M_-$ but not $M_+$ and adds cones $\sigma_K$, where $K=(I\cup M_+)\setminus J$ for any non-empty subset $J\subset M_-$.
\end{remark}

By \cite{CIJ}*{Proposition 5.5}, the loci of indeterminacy of $\phi$ and $\phi^{-1}$ are the toric substacks
\begin{align}\label{indeterminacy-loci}
\bigcap_{j\in M_-}\{z_j=0\}\subset X_+, \quad \text{and } \bigcap_{j\in M_+}\{z_j=0\}\subset X_-.
\end{align}

Following \cite{CIJ}, there is a common toric blow-up $\tilde X$ of $X_+$ and $X_-$ that fits into the commutative diagram
\[
\begin{tikzcd}
 & \arrow[dl,swap]{}{f_+} \tilde X \arrow[dr]{}{f_-} \\
X_+  \arrow[rr,dashed]{}{\phi} && X_-.
\end{tikzcd}
\]
where $f_{\pm}:\tilde{X}\rightarrow X_{\pm}$ are projective birational. 
The blow-up $\tilde X$ is obtained as follows. Recall that we have
\[
\sum_{i=1}^m (D_i\cdot e)b_i=0
\]
by the fan sequence (\ref{fan-seq}). We consider the vector
\[
b_{m+1}=\sum_{i\in M_+} (D_i\cdot e)b_i=\sum_{i\in M_-}-(D_i\cdot e)b_i.
\]
Now, consider the fan sequence
\begin{equation}\label{fan-seq-blow-up}
0 \longrightarrow \mathbb{L}\oplus \mathbb Z \longrightarrow \mathbb{Z}^{m+1} \stackrel{\tilde{\beta}}{\longrightarrow} N\longrightarrow 0,
\end{equation}
such that the map $\tilde{\beta}: \mathbb Z^{m+1}\rightarrow N$ is given by $\{\beta_1,\ldots,\beta_m,\beta_{m+1}\}$ and the map $\mathbb L\oplus \mathbb Z\rightarrow \mathbb Z^{m+1}$ is given by $(\tilde {D}_1,\ldots,\tilde {D}_m,\tilde{D}_{m+1})$ where
\[   
\tilde D_i=\left\{
\begin{array}{ll}
      D_i\oplus 0 & \text{if $i\in\{1,\ldots,m\}$ and $D_i\cdot e\leq 0$} \\
      D_i\oplus (-D_i\cdot e) & \text{if $i\in\{1,\ldots,m\}$ and $D_i\cdot e> 0$} \\
      0\oplus 1 & \text{If $i=m+1$}.
\end{array} 
\right. \]
Let $\tilde{C}$ be the chamber containing the stability condition $\tilde{\omega}:=(\omega_0,-\varepsilon)$, where $\omega_0$ is in the relative interior of $W\cap \overline{C_+}=W\cap \overline{C_-}$ and $\varepsilon$ is a positive and sufficiently small real number. Then $\tilde X$ is the toric Deligne--Mumford stack given by the stability condition $\tilde{\omega}$. Note that the stacky fan of $\tilde X$ is obtained from the stacky fans of $X_+$ and $X_-$ by adding the extra ray $b_{m+1}$. 

The toric morphisms $f_{\pm}:\tilde {X}\rightarrow X_{\pm}$ are induced by natural maps of stacky fans. We also write $E:=\tilde {D}_{m+1}$ for the toric divisor of $\tilde X$ corresponding to the extra ray $b_{m+1}$.  It follows from \cite{CIJ}*{Proposition 6.21} that
\[
K_{\tilde X}=f_-^*K_{X_-}+\left(1+\sum_{i\in M_-} D_i\cdot e\right)E=f_+^*K_{X_+}+\left(1-\sum_{i\in M_+} D_i\cdot e\right)E.
\]
Then we have the following relation
\begin{lemma}[\cite{Iritani20}, Lemma 5.1]
Let $K_{X_{\pm}}$ be the canonical class of $X_{\pm}$, then we have
\[
f_+^*K_{X_+}=f_-^*K_{X_-}+\left(\sum_{i=1}^m D_i\cdot e\right)E.
\]
\end{lemma}

We define
\[
\mathbb K_\pm:=\left\{f\in \mathbb L\otimes \mathbb Q|\left\{i \in \{1,\ldots,m\}| D_i\cdot f\in \mathbb Z \right\}\in \mathcal A_\pm\right\}.
\]
Let $\tilde{\mathbb L}_\pm$ be the free $\mathbb Z$-submodule of $\mathbb L\otimes \mathbb Q$ generated by $\mathbb K_\pm$ and set
\[
\tilde{\mathbb L}^\vee_\pm=\Hom(\tilde(\mathbb L)_\pm,\mathbb Z)\subset \mathbb L^\vee.
\]
By \cite{CIJ}*{Lemma 5.8}, there is a decomposition
\[
\tilde{\mathbb L}^\vee_\pm=\left(H^2(X_\pm;\mathbb R)\cap \tilde{\mathbb L}^\vee_\pm\right)\oplus \bigoplus_{j\in S_\pm}\mathbb Z D_j.
\]
\if{
Set 
\[
l_\pm=\dim H^2(X_\pm;\mathbb R)=r-\#(S_\pm) \quad \text{and } l=r-1-\#(S_0).
\]
Note that $l\leq l_\pm$. We can choose an integral bases of $\tilde{\mathbb L}^\vee_\pm$:
\[
\{p_1^{\pm},\ldots,p_{l_\pm}^\pm\}\cup \{D_j|j\in S_\pm\}\subset \tilde{\mathbb L}^\vee_\pm,
\]
such that
\begin{itemize}
    \item $p_i^\pm$ lies in $\overline{C^\prime_\pm}\subset H^2(X_\pm;\mathbb R)$ for $1\leq i \leq l_\pm$;
    \item $p_i^+=p_i^-\in \overline{C^\prime_W}$ for $i=1,\ldots, l$.
\end{itemize}

For $d\in \mathbb L$, let $y^d$ denote the corresponding element in $\mathbb C[\mathbb L]$. We have an inclusion
\[
\mathbb C[C_+^\vee\cap \mathbb L]\hookrightarrow \mathbb C[y_1,\ldots,y_{l_+},\{x_j\}_{j\in S_+}]
\]
given by
\[
y^d \mapsto \prod_{i=1}^{l_+}y_i^{p_i^+\cdot d}\prod_{j\in S_+}x_j^{D_j\cdot d}.
\]
Similarly, we have another inclusion
\[
\mathbb C[C_-^\vee\cap \mathbb L]\hookrightarrow \mathbb C[\tilde y_1,\ldots,\tilde y_{l_-},\{\tilde x_j\}_{j\in S_-}]
\]
given by
\[
y^d \mapsto \prod_{i=1}^{l_-}\tilde y_i^{p_i^-\cdot d}\prod_{j\in S_-}\tilde x_j^{D_j\cdot d}.
\]
}\fi

We can choose integral bases of $\tilde{\mathbb L}^\vee_\pm$:
\[
\{p_1^{\pm},\ldots,p_{\mathrm{r}}^\pm\}\subset \tilde{\mathbb L}^\vee_\pm,
\]
such that
\begin{itemize}
    \item $p_i^\pm$ lies in $\overline{C^\prime_\pm}\subset H^2(X_\pm;\mathbb R)$ for $1\leq i \leq \mathrm{r}$;
    \item $p_i^+=p_i^-\in \overline{C^\prime_W}$ for $i=1,\ldots, \mathrm{r}-1$.
\end{itemize}
For $d\in \mathbb L$, let $y^d$ denote the corresponding element in $\mathbb C[\mathbb L]$. We have an inclusion
\[
\mathbb C[C_+^\vee\cap \mathbb L]\hookrightarrow \mathbb C[y_1,\ldots,y_{\mathrm{r}}]
\]
given by
\[
y^d \mapsto \prod_{i=1}^{\mathrm{r}}y_i^{p_i^+\cdot d}.
\]
Similarly, we have another inclusion
\[
\mathbb C[C_-^\vee\cap \mathbb L]\hookrightarrow \mathbb C[\tilde y_1,\ldots,\tilde y_{\mathrm{r}}]
\]
given by
\[
y^d \mapsto \prod_{i=1}^{\mathrm{r}}\tilde y_i^{p_i^-\cdot d}.
\]
The coordinates are related by change of variables
\begin{align}\label{change-of-variables}
\tilde y_i=\left\{\begin{array}{cc}
    y_iy_\mathrm{r}^{c_i} & 1\leq i \leq \mathrm{r}-1 \\
    y_\mathrm{r}^{-c} & i=\mathrm{r} 
\end{array}
\right.
\end{align}
where $c=-p_\mathrm{r}^+\cdot e/p_\mathrm{r}^-\cdot e\in \mathbb Q_{>0}$ and $c_i\in \mathbb Q$ are determined by the change of basis from $\{p_i^+\}$ to $\{p_i^-\}$. 

\section{Relative quantum cohomology for toric pairs}\label{sec:toric-pairs}

In this section we consider the divisors $D\pm\subset X_\pm$, where
\[
D_\pm=\sum_{i\in (M_+\cup M_-)\setminus S_\pm} \bar{D}_i.
\]
Then, by (\ref{indeterminacy-loci}), the loci of indeterminacy of $\phi$ and $\phi^{-1}$ are complete intersections of the irreducible components of $D_+$ and $D_-$. Furthermore,
\begin{lemma}
We have
\[
f_+^*(K_{X_+}+D_+)=f_-^*(K_{X_-}+D_-)
\]
\end{lemma}
\begin{proof}
This follows from \cite{CIJ}*{Proposition 6.21}.
\end{proof}

For each type of toric wall-crossing, we have
\begin{enumerate}
    \item[(I)] Flip/flop: 
    \[
    D_+=\sum_{i\in (M_+\cup M_-)} \bar{D}_i\subset X_+ \text{ and }D_-=\sum_{i\in (M_+\cup M_-)} \bar{D}_i\subset X_-.
    \]
    \item[(II)] Discrepant/crepant resolution: If $S_-=S_+\sqcup \{j\}$, $M_-=\{j\}$ and $\#\{M_+\}\geq 2$, then 
    \[
    D_+=\sum_{i\in (M_+\cup \{j\})} \bar{D}_i\subset X_+  \text{ and }D_-=\sum_{i\in M_+} \bar{D}_i\subset X_-.
    \]
    If $S_+=S_-\sqcup \{j\}$, $M_+=\{j\}$ and $\#\{M_-\}\geq 2$, then 
    \[
    D_+=\sum_{i\in M_-} \bar{D}_i\subset X_+ \text{ and }D_-=\sum_{i\in (M_-\cup \{j\})} \bar{D}_i\subset X_-.
    \]
    \item[(III)] root/gerbe construction: 
    \[
    D_+=\bar{D}_{i_-} \text{ and } D_-=\bar{D}_{i_+}.
    \]
\end{enumerate}

\subsection{The $I$-functions}
A Givental style mirror theorem for toric stacks is proved in \cite{CCIT15} and \cite{CCFK}. The $I$-function for a toric stack is the following.

\begin{defn}
The $I$-function of a toric stack $ X$ is an $H^*_{\on{CR}}(X)$-valued power series defined by
\[
I_{X}(y,z)=ze^{t/z}\sum_{d\in\mathbb K}y^{d}\left(\prod_{i=0}^{m}\frac{\prod_{a\leq 0,\langle a\rangle=\langle D_i\cdot d\rangle}(\bar{D}_i+az)}{\prod_{a \leq D_i\cdot d,\langle a\rangle=\langle D_i\cdot d\rangle}(\bar{D}_i+az)}\right)\textbf{1}_{[-d]},
\]
where $t=\sum_{a=1}^\mathrm{r} \bar p_a \log y_a$, $y^d=y_1^{p_1\cdot d}\cdots y_\mathrm{r}^{p_\mathrm{r}\cdot d}$ and, $[-d]$ is the equivalence class of $-d$ in $\mathbb K/ \mathbb L$.
\end{defn}

Then $I$-functions for $X_+$ and $X_-$ are
\[
I_{X_+}(y,z)=ze^{t_+/z}\sum_{d\in\mathbb K_{+}}y^{d}\left(\prod_{i=0}^{m}\frac{\prod_{a\leq 0,\langle a\rangle=\langle D_i\cdot d\rangle}(\bar{D}_i+az)}{\prod_{a \leq D_i\cdot d,\langle a\rangle=\langle D_i\cdot d\rangle}(\bar{D}_i+az)}\right)\textbf{1}_{[-d]},
\]
and 
\[
I_{X_-}(y,z)=ze^{t_-/z}\sum_{d\in\mathbb K_{-}}\tilde y^{d}\left(\prod_{i=0}^{m}\frac{\prod_{a\leq 0,\langle a\rangle=\langle D_i\cdot d\rangle}(\bar{D}_i+az)}{\prod_{a \leq D_i\cdot d,\langle a\rangle=\langle D_i\cdot d\rangle}(\bar{D}_i+az)}\right)\textbf{1}_{[-d]},
\]
where 
\[
t_+=\sum_{a=1}^\mathrm{r} \bar p_a^+ \log y_a, \quad y^d=y_1^{p_1^+\cdot d}\cdots y_\mathrm{r}^{p_\mathrm{r}^+\cdot d}
\]
and 
\[
t_-=\sum_{a=1}^\mathrm{r} \bar p_a^- \log \tilde y_a, \quad \tilde y^d=\tilde y_1^{p_1^-\cdot d}\cdots \tilde y_\mathrm{r}^{p_\mathrm{r}^-\cdot d}.
\]
Set $I_+=(M_+\cup M_-)\setminus S_+$ and $I_-=(M_+\cup M_-)\setminus S_-$. 
Now, we consider relative $I$-functions for $(X_+,D_+)$ and $(X_-,D_-)$:
\begin{align*}
I_{(X_+,D_+)}(y,z)=ze^{t_+/z}\sum_{d\in\mathbb K_{+}}y^{d}\left(\prod_{i\in M_0 }\frac{\prod_{a\leq 0,\langle a\rangle=\langle D_i\cdot d\rangle}(\bar{D}_i+az)}{\prod_{a \leq D_i\cdot d,\langle a\rangle=\langle D_i\cdot d\rangle}(\bar{D}_i+az)}\right)\textbf{1}_{[-d]}I_{D_+,d},
\end{align*}
and 
\[
I_{(X_-,D_-)}(y,z)=ze^{t_-/z}\sum_{d\in\mathbb K_{-}}\tilde y^{d}\left(\prod_{i\in M_0}\frac{\prod_{a\leq 0,\langle a\rangle=\langle D_i\cdot d\rangle}(\bar{D}_i+az)}{\prod_{a \leq D_i\cdot d,\langle a\rangle=\langle D_i\cdot d\rangle}(\bar{D}_i+az)}\right)\textbf{1}_{[-d]}I_{D_-,d},
\]
where
\[
I_{D_+,d}=\frac{1}{\prod_{i\in I_+, D_i\cdot d>0}(\bar D_i+(D_i\cdot d)z)}[\textbf{1}]_{(-D_i\cdot d)_{i\in I_+}},
\]
and
\[
I_{D_-,d}=\frac{1}{\prod_{i\in I_-, D_i\cdot d>0}(\bar D_i+(D_i\cdot d)z)}[\textbf{1}]_{(-D_i\cdot d)_{i\in I_-}}.
\]

\begin{remark}\label{rmk-extended-data}
The extended data for the pairs $(X_+,D_+)$ and $(X_-,D_-)$ that we choose here is slightly different from the extended data for $X_+$ and $X_-$. In order to match the data, we choose their extended data to be $S_0$. For example, if we have a Type II wall-crossing with $S_-=S_+\sqcup \{j\}$, then the extended data for $(X_-,D_-)$ is $S_-\setminus \{j\}=S_+=S_0\subset M_0$ instead of $S_-$. 
\end{remark}

\begin{theorem}[\cite{TY20c}, Theorem 29]
The $I$-function $I_{(X_\pm,D_\pm)}$ lies in Givental's Lagrangian cone for $(X_\pm,D_\pm)$ defined in \cite{TY20c}*{Section 5}.
\end{theorem}

\begin{remark}
When the birational transformation  $\phi: \xymatrix{
X_+\ar@{-->}[r]& X_-
}$ is discrepant and $\sum_{i=1}^m D_i\cdot e>0$, the $I$-function $I_{X_+}$ is analytic with respect to $y_\mathrm{r}$ with radius of convergence $\infty$, but the $I$-function $I_{X_-}$ is not analytic in the $\tilde y_\mathrm{r}$ variable (see \cite{AS}*{Section 4.1}). By considering the pairs $(X_+,D_+)$ and $(X_-,D_-)$ instead, the relative $I$-functions $I_{(X_+,D_+)}$ and $I_{(X_-,D_-)}$ are both analytic. In \cite{AS}*{Section 4.1}, the authors considered the regularized $I$-function of $I_{X_-}$ as a replacement of the $I$-function $I_{X_-}$. The regularized $I$-function was constructed such that the asymptotic expansion of the Laplace transform of the regularized $I$-function is given by the $I$-function $I_{X_-}$. But the direct enumerative meaning of the regularized $I$-functions was never discussed.  While the relative $I$-functions are different from the regularized $I$-functions, one may consider the relative $I$-functions as an alternative to the absolute $I$-functions such that one can study analytic continuation of the relative $I$-functions. In fact, with the choice of the divisors in this section, the analytic continuation is not even necessary. In addition, the enumerative meaning of the relative $I$-functions is clear: they are mirrors of the corresponding generating functions of genus zero Gromov--Witten theories of log crepant pairs.
\end{remark}

\subsection{The $H$-functions}
\begin{defn}\label{defn-gamma-class}
We define the $\hat{\Gamma}$-class of $(X_+,D_+)$ as
\[
\hat{\Gamma}_{(X_+,D_+)}:=\bigoplus_{\vec d\in (\mathbb Z)^{|I_+|},f\in \mathbb K_{+}/\mathbb L} \prod_{i\in M_0}\Gamma(1+\bar D_i-\langle D_i\cdot f\rangle)\prod_{i\in I_+,d_i<0}\frac{1}{\bar D_i-d_i}\textbf{1}_{f}[\textbf 1]_{(d_i)_{i\in I_+}}.
\]
\end{defn}

\begin{remark}
  The Gromov--Witten theory that we consider here is the formal Gromov--Witten theory for infinite root stacks defined in \cite{TY20c}. It is a limit of the orbifold Gromov--Witten theory of multi-root stacks.  The $\hat{\Gamma}$-class defined in Definition \ref{defn-gamma-class} is also taken as a limit of the $\hat{\Gamma}$-class for orbifold Gromov--Witten theory in \cite{Iritani09}.
\end{remark}

Note that the $I$-function can be written in terms of $\Gamma$-functions:
\begin{align*}
I_{(X_+,D_+)}(y,z)=&ze^{t_+/z}\sum_{d\in\mathbb K_{+}}\frac{y^{d}}{z^{(\sum_{i\in M_0}D_i)\cdot d}}\left(\prod_{i\in M_0}\frac{\Gamma(1+\frac{\bar D_i}{z}-\langle -D_i\cdot d\rangle )}{\Gamma (1+\frac{\bar D_i}{z}+D_i\cdot d)}\right)\\
&\quad \cdot\left( \frac{1}{\prod_{i\in I_+, D_i\cdot d>0}(\frac{\bar D_i}{z}+D_i\cdot d)} \right)\frac{\textbf{1}_{[-d]}}{z^{\iota_{[-d]}+\#\{i\in I_+,D_i\cdot d>0 \}}}[\textbf{1}]_{(-D_i\cdot d)_{i\in I_+}}
\end{align*}

We define the $H$-function as
\begin{align}\label{H-function}
H_{(X_+,D_+)}(y)=e^{\frac{t_+}{2\pi i}}\sum_{d\in\mathbb K_{+}}y^d\left(\frac{1}{\prod_{i\in M_0}\Gamma (1+\frac{\bar D_i}{2\pi i}+D_i\cdot d)}\right)\textbf{1}_{[d]}[\textbf{1}]_{(D_i\cdot d)_{i\in I_+}}.
\end{align}
Then the relation between the $I$-function and the $H$-function is
\begin{align}\label{H-I}
z^{-1}I_{(X_+,D_+)}(y,z)=z^{-\frac{\dim X_+}{2}}z^{-\mu^+}z^{\rho^+}\left(\hat{\Gamma}_{(X_+,D_+)}\cup (2\pi i)^{\frac{\deg_0}{2}}\on{inv}^*H(z^{-\frac{\deg y}{2}}y)\right),
\end{align}
where
\begin{itemize}
    \item $\rho:=c_1(T_X(-\log D))\in H^2(X)$;
    \item $\mu$ is the grading operator (\ref{grading-operator}) for the pair $(X,D)$;
    \item $z^{-\mu}z^{\rho}:=\exp(-\mu\log z)\exp(\rho \log z)$;
    \item $\deg_0:\mathfrak H_{\vec s}\rightarrow \mathfrak H_{\vec s}$ is the degree operator defined by
    \[
    \deg_0(\phi)=2p\phi, \text{ for }\phi\in H^{2p}(D_{I_{\vec s}});
    \]
    \item $\on{inv}^*: \mathfrak H \rightarrow \mathfrak H$ is the involution such that $\on{inv}^*[\alpha]_{\vec s}=[\alpha]_{-\vec s}$, for $[\alpha]_{\vec s}\in \mathfrak H_{\vec s}$. And, $\on{inv}^*\textbf{1}_{[d]}=\textbf{1}_{[-d]}$.
\end{itemize}


Similarly, for $(X_-,D_-)$.

We write
\[
H_{(X,D)}(y)=\sum_{d\in\mathbb K}\sum_{\vec d=(D_i\cdot d)_{i\in I}\in (\mathbb Z)^{|I|},f=[d]\in \mathbb K/\mathbb L}H_{(X,D),f,\vec d}\textbf{1}_{f}[\textbf{1}]_{\vec d}.
\]
The $H$-functions of $(X_+,D_+)$ and $(X_-,D_-)$ are easily identified following directly from the expression of the $H$-function in (\ref{H-function}). 
\begin{theorem}\label{thm-toric-stack-snc}
For any $d_+\in \mathbb K_+$, $d_-\in \mathbb K_-$ such that $d_+-d_-\in \mathbb Q e$, let $f_\pm=[d_\pm]$ and $\vec d_\pm=(D_i\cdot d_\pm)_{i\in I_\pm}$. We have
\[
H_{(X_+,D_+),f_+,\vec d_+}=H_{(X_-,D_-),f_-,\vec d_-} 
\]
with the change of variable (\ref{change-of-variables}). 
\end{theorem}



\begin{remark}
Recall that the extended data for $(X_\pm,D_\pm)$ may be different from the extended data for $X_\pm$ as mentioned in Remark \ref{rmk-extended-data}. Hence, when comparing the $I$-functions (or $H$-functions), some of the variables may need to be set to zero based on how the extended data changed. See Section \ref{sec:example-blow-up} for an example.
\end{remark}

\begin{remark}
  Note that there is neither analytic continuation of \cite{CIJ} nor asymptotic expansion of \cite{AS} needs to be done. Furthermore, the relation holds for both crepant and discrepant cases. Of course, the discrepant resolution case is slightly different since the correspondence between the components of their $H$-functions is not a one-to-one correspondence because there are more curves classes on one side than the other side.
\end{remark}

\begin{remark}
Even in the crepant case, the relation between relative quantum cohomology is also significantly simpler than the relation between absolute quantum cohomology. Recall that, the relation between absolute quantum cohomology involves analytic continuation and complicated symplectic transformations \cite{CIJ}. One can avoid all these complexities by considering relative quantum cohomology with the choice of divisors in this section.

The analytic continuation in \cite{CIJ} concerns the poles of the $\Gamma$-functions of $\bar D_i$ for $i\not\in M_0$. These $\Gamma$-functions become $1$ in our setting. Hence, the analytic continuation is not presented here. 
\end{remark}

\subsection{Examples}

\subsubsection{Root construction}
Given a smooth projective variety $X$, a smooth divisor $D\subset X$ and a positive integer $r$, the natural map from the $r$-th root stack $X_{D,r}$ to $X$: 
\[
X_{D,r}\rightarrow X
\]
is a discrepant birational transformation. Let $\mathcal D_r\subset X_{D,r}$ be the $r$-th root of $D$. In our setting, we are comparing relative Gromov--Witten invariants of $(X_{D,r},\mathcal D_r)$ and relative Gromov--Witten invariants of $(X,D)$. By \cite{AF}*{Proposition 4.5.1}, relative Gromov--Witten invariants of $(X_{D,r},\mathcal D_r)$ are equal to relative Gromov--Witten invariants of $(X,D)$. This is the case that is already known for smooth projective varieties and does not require the machinery from previous sections. This is significantly simpler than directly comparing the Gromov--Witten theory of $X_{D,r}$ and the Gromov--Witten theory of $X$ (see, for example \cite{TY16}.).

\subsubsection{ $\mathbb P^2$ and $F_1$}\label{sec:example-blow-up}
The Hirzebruch surface $F_1$ is a blow-up of $\mathbb P^2$ at a point. Therefore, we can compare Gromov--Witten invariants of $\mathbb P^2$ relative to two lines and Gromov--Witten invariants of $F_1$ relative to the exceptional divisor of the blow-up and the strict transforms of the two lines. Recall that the $I$-function for $\mathbb P^2$ is
\[
I_{\mathbb P^2}(y,z)=ze^{H\log y /z}\sum_{d\geq 0}y^{d}\left(\frac{1}{\prod_{a=1}^d(H+az)^3}\right),
\]
where $H\in H^2(\mathbb P^2)$ is the hyperplane class.

Now, consider the Hirzebruch surface $F_1=\mathbb P(\mathcal O_{\mathbb P^1}\oplus \mathcal O_{\mathbb P^1}(-1))$. Let $P,H$ be nef basis of $H^2(F_1)$ that are Poincar\'e dual to the fiber and the infinity section. The divisor matrix is given by
\[
\begin{bmatrix}
0 & -1 & 1 &1\\
1 & 1 & 0 & 0
\end{bmatrix}.
\]
Then
\[
D_1=H, \quad D_2=H-P, \quad D_3=D_4=P.
\]
The $I$-function for $F_1$ is
\[
I_{F_1}(y_1,y_2,z)=ze^{(P\log y_1+ H\log y_2)/z}\sum_{d_1,d_2\geq 0} y_1^{d_1}y_2^{d_2}\frac{\prod_{a=-\infty}^0 (H-P+az)}{\prod_{a=1}^{d_2}(H+az)\prod_{a=-\infty}^{d_2-d_1}(H-P+az) \prod_{a=1}^{d_1}(P+az)^2}.
\]
Then we compare relative quantum cohomology of $(\mathbb P^2,H+H)$ and $(F_1,D_2+D_3+D_4)$. Then the $H$-functions are
\[
H_{(\mathbb P^2,H+H)}(y)=e^{\frac{H\log y}{2\pi i}}\sum_{d\geq 0} y^d \frac{1}{\Gamma (1+\frac{H}{2\pi i}+d)}[\textbf{1}]_{d,d},
\]
and 
\[
H_{(F_1,D_2+D_3+D_4)}(y_1,y_2)=e^{\frac{P\log y_1+ H\log y_2}{2\pi i}}\sum_{d_1,d_2\geq 0}y_1^{d_1}y_2^{d_2} \frac{1}{\Gamma(1+\frac{H}{2\pi i}+d_2)}[\textbf{1}]_{d_2-d_1,d_1,d_1}.
\]
For each $d\geq 0$, we have 
\[
H_{(\mathbb P^2,H+H),(d,d)}(y)=H_{(F_1,D_2+D_3+D_4),(d,0,0)}(0,y).
\]



\subsubsection{local model for simple flips}
Quantum cohomology under simple flips has been studied in \cite{LLW} for local models. It was proved  in \cite{LLW} that quantum cohomology rings are related by analytic continuation in this case. One considers the local models for simple $(r,r^\prime)$ flips with $r>r^\prime>0$. In other words,
\[
X_+=\mathbb P_{\mathbb P^r}(\mathcal O(-1)^{r^\prime+1}\oplus \mathcal O) \text{ and } X_-=\mathbb P_{\mathbb P^{r^\prime}}(\mathcal O(-1)^{r+1}\oplus \mathcal O).
\]
If we use the set-up in Section \ref{sec:toric-pairs}, then we need to take the divisors $D_\pm$ to be the full toric boundaries. Then the identification of their relative $I$-functions are just trivial. We can also just take part of their toric boundaries. Let $h_+$ (resp. $h_-$) be the hyperplane class of $\mathbb P^r$ (resp. $\mathbb P^{r^\prime}$) and $\xi_+$ (resp. $\xi_-$) is the hyperplane class of $X_+\rightarrow \mathbb P^r$ (resp. $X_-\rightarrow \mathbb P^{r^\prime}$). Let
\[
D_+=\underbrace{h_++\cdots +h_+}_{r-r^\prime}, \text{ and } D_-=\underbrace{(\xi_--h_-)+\cdots+(\xi_--h_-)}_{r-r^\prime}.
\]
Then the $H$-functions are
\begin{align*}
&H_{(X_+,D_+)}(y_1,y_2)\\
=&e^{\frac{\xi_+\log y_1+ h_+\log y_2}{2\pi i}}\sum_{d_1,d_2\geq 0} \frac{y_1^{d_1}y_2^{d_2}[\textbf{1}]_{(d_2,\cdots,d_2)}}{\Gamma(1+\frac{h_+}{2\pi i}+d_2)^{r^\prime+1}\Gamma(1+\frac{\xi_+-h_+}{2\pi i}+d_1-d_2)^{r^\prime+1}\Gamma(1+\frac{\xi_+}{2\pi i}+d_1)}.
\end{align*}
\begin{align*}
&H_{(X_-,D_-)}(y_1,y_2)\\
=&e^{\frac{\xi_-\log y_1+ h_-\log y_2}{2\pi i}}\sum_{d_1,d_2\geq 0} \frac{y_1^{d_1}y_2^{d_2}[\textbf{1}]_{(d_1-d_2,\cdots,d_1-d_2)}}{\Gamma(1+\frac{h_-}{2\pi i}+d_2)^{r^\prime+1}\Gamma(1+\frac{\xi_--h_-}{2\pi i}+d_1-d_2)^{r^\prime+1}\Gamma(1+\frac{\xi_-}{2\pi i}+d_1)}.
\end{align*}
It is straightforward to see that their $H$-functions are identical
\[
H_{(X_+,D_+),(d_2,\cdots,d_2)}=H_{(X_-,D_-),(d_1-d_2,\cdots,d_1-d_2)}
\]
under the morphism
\[
\Phi(h_+)=\xi_--h_-, \quad \Phi(\xi_+)=\xi_-
\]
described in \cite{LLW10}*{Section 2.3}.

\subsection{The extended $I$-functions}\label{sec:extended-I-function}
In previous sections, we consider the $I$-functions for pairs $(X_+,D_+)$ and $(X_-,D_-)$ with extended data only from the orbifold structures of $X_\pm$. They are considered as non-extended $I$-functions for relative theories. There are also the extended $I$-functions for relative theories considered in \cite{FTY} and \cite{TY20b}. It is straightforward to see that the same conclusion holds for the extended $I$-functions. For example, if we consider the components of the $I$-functions that take value in $H^*(X_\pm)$, then we can see the identification of their $I$-functions directly. Recall that, the part of extended $I$-functions for pairs $(X_+,D_+)$ and $(X_-,D_-)$ that take value in $H^*(X_\pm)$ are
\begin{align*}
I_{(X_+,D_+),0}(y,x,z)=ze^{t_+/z}\sum_{\substack{d\in \mathbb K_+,(k_{i1},\ldots,k_{im})\in (\mathbb Z_{\geq 0})^m\\ \sum_{j=1}^m a_{ij}k_{ij}= d_i, i\in I_+}  }y^{d} & \frac{\prod_{i\in I_+}\prod_{j=1}^m x_{ij}^{k_{ij}}}{z^{\sum_{i\in I_+}\sum_{j=1}^m k_{ij}}\prod_{i\in I_+}\prod_{j=1}^m(k_{ij}!)}\\
&\left(\prod_{i\in M_0 }\frac{\prod_{a\leq 0,\langle a\rangle=\langle D_i\cdot d\rangle}(\bar{D}_i+az)}{\prod_{a \leq D_i\cdot d,\langle a\rangle=\langle D_i\cdot d\rangle}(\bar{D}_i+az)}\right),
\end{align*}
and 
\begin{align*}
I_{(X_-,D_-),0}(\tilde{y}, \tilde{x},z)=ze^{t_-/z}\sum_{\substack{d\in \mathbb K_-,(k_{i1},\ldots,k_{im})\in (\mathbb Z_{\geq 0})^m\\ \sum_{j=1}^m a_{ij}k_{ij}= d_i, i\in I_-} }\tilde y^{d}& \frac{\prod_{i\in I_-}\prod_{j=1}^m \tilde x_{ij}^{k_{ij}}}{z^{\sum_{i\in I_-}\sum_{j=1}^m k_{ij}}\prod_{i\in I_-}\prod_{j=1}^m(k_{ij}!)}\\
&\left(\prod_{i\in M_0}\frac{\prod_{a\leq 0,\langle a\rangle=\langle D_i\cdot d\rangle}(\bar{D}_i+az)}{\prod_{a \leq D_i\cdot d,\langle a\rangle=\langle D_i\cdot d\rangle}(\bar{D}_i+az)}\right).
\end{align*}

One can set $x_{ij}=\tilde x_{ij}=1$ and $z=1$. Then the extended $I$-functions are identical up to factors of $k_{ij}!$,
under the change of variable (\ref{change-of-variables}). 


\section{Toric complete intersections}\label{sec:toric-ci}

Now we consider how relative quantum cohomology for toric complete intersections changes under birational transformations. 
\subsection{The standard set-up}\label{sec:complete-intersection-1}
Given a toric birational transformation $\phi: \xymatrix{
X_+\ar@{-->}[r]& X_-
}$
as in Section \ref{sec:toric-pairs}. Consider characters $E_1,\ldots, E_k\in \mathbb L^\vee$. We obtain the corresponding line bundles $L_{1,\pm},\ldots, L_{k,\pm}$ over $X_{\pm}$. Let
\[
E_+:=\bigoplus_{i=1}^k L_{i,+}, \quad \text{and} \quad E_-:= \bigoplus_{i=1}^k L_{i,-}.
\]
Let $s_+$ and $s_-$ be regular sections of the vector bundles $E_+\rightarrow X_+$ and $E_-\rightarrow X_-$ respectively.
Following \cite{CIJ}*{Section 7}, we assume that 
\begin{itemize}
    \item $E_i$ lies in $\overline{C_W}$ for $i=1,2,\ldots, k$;
    \item $L_{i,\pm}$ is pulled back from the coarse moduli space $|X_\pm|$;
    \item $s_+$ and $s_-$ are compatible via $\phi: \xymatrix{
X_+\ar@{-->}[r]& X_-
}$.
\end{itemize}
Let $Y_+\subset X_+$ and $Y_-\subset X_-$ be the complete intersections defined by $s_+$ and $s_-$ respectively. The birational map $\phi$ induces a birational map $\phi_Y: \xymatrix{
Y_+\ar@{-->}[r]& Y_-
}$.
Let $D_{Y,\pm}=D_\pm\cap Y_\pm=\sum_{i\in (M_+\cup M_-)\setminus S_\pm} \bar D_i\cap Y_\pm$. 

Then we can compare the ambient parts of the relative quantum cohomology of $(Y_+,D_{Y,+})$ and $(Y_-,D_{Y,-})$ pulled back from the corresponding relative quantum cohomology of the ambient toric pairs $(X_+,D_+)$ and $(X_-,D_-)$. Let $v_i$ be the first Chern class of the line bundle corresponding to the character $E_i$ for $1\leq i \leq k$. Using the orbifold quantum Lefschetz principle of \cite{Tseng}, we obtain the relative $I$-functions for $(Y_+,D_{Y,+})$ and $(Y_-,D_{Y,-})$ as follows.
\begin{align*}
I_{(Y_+,D_{Y,+})}(y,z)=ze^{t_+/z}\sum_{d\in\mathbb K_{+}}y^{d}\left(\prod_{i\in M_0 }\frac{\prod_{a\leq 0,\langle a\rangle=\langle D_i\cdot d\rangle}(\bar{D}_i+az)}{\prod_{a \leq D_i\cdot d,\langle a\rangle=\langle D_i\cdot d\rangle}(\bar{D}_i+az)}\right)\\
\quad \left(\prod_{i=1}^k\prod_{0<a\leq v_i\cdot d}(v_i+az)\right)\textbf{1}_{[-d]}I_{D_{Y,+},d},
\end{align*}
where
\[
I_{D_{Y,+},d}=\frac{1}{\prod_{i\in I_+, D_i\cdot d>0}(\bar D_i+(D_i\cdot d)z)}[\textbf{1}]_{(-D_i\cdot d)_{i\in I_+}};
\]
and
\begin{align*}
I_{(Y_-,D_{Y,-})}(y,z)=ze^{t_-/z}\sum_{d\in\mathbb K_{-}}\tilde y^{d}\left(\prod_{i\in M_0}\frac{\prod_{a\leq 0,\langle a\rangle=\langle D_i\cdot d\rangle}(\bar{D}_i+az)}{\prod_{a \leq D_i\cdot d,\langle a\rangle=\langle D_i\cdot d\rangle}(\bar{D}_i+az)}\right)\\
\quad \left(\prod_{i=1}^k\prod_{0<a\leq v_i\cdot d}(v_i+az)\right)\textbf{1}_{[-d]}I_{D_{Y,-},d},
\end{align*}
where
\[
I_{D_{Y,-},d}=\frac{1}{\prod_{i\in I_-, D_i\cdot d>0}(\bar D_i+(D_i\cdot d)z)}[\textbf{1}]_{(-D_i\cdot d)_{i\in I_-}}.
\]

We define the $\hat{\Gamma}$-class of $(Y_+,D_{Y,+})$ as
\[
\hat{\Gamma}_{(Y_+,D_{Y,+})}:=\bigoplus_{\vec d\in (\mathbb Z)^{|I_+|},f\in \mathbb K_{+}/\mathbb L} \frac{\prod_{i\in \bar{I}_+}\Gamma(1+\bar D_i-\langle D_i\cdot f\rangle)}{\prod_{i=1}^k\Gamma(1+v_i)}\prod_{i\in I_+,d_i<0}\frac{1}{\bar D_i-d_i}\textbf{1}_{f}[\textbf 1]_{(d_i)_{i\in I_+}}.
\]

We define the $H$-function as
\begin{align}
H_{(Y_+,D_{Y,+})}(y)=e^{\frac{t_+}{2\pi i}}\sum_{d\in\mathbb K_{+}}y^d\left(\frac{
\prod_{i=1}^k \Gamma(1+\frac{v_i}{z}+v_i\cdot d)
}{\prod_{i\in M_0}\Gamma (1+\frac{\bar D_i}{z}+D_i\cdot d)}\right)\textbf{1}_{[d]}[\textbf{1}]_{(D_i\cdot d)_{i\in I_+}}.
\end{align}
Then the relation between the $I$-function and the $H$-function is
\begin{align}
z^{-1}I_{(Y_+,D_{Y,+})}(y,z)=z^{-\frac{\dim Y_+}{2}}z^{-\mu^+}z^{\rho^+_Y}\left(\hat{\Gamma}_{(Y_+,D_{Y,+})}\cup (2\pi i)^{\frac{\deg_0}{2}}\on{inv}^*H(z^{-\frac{\deg y}{2}}y)\right),
\end{align}
where $\rho_Y=c_1(T_X(-\log D))+c_1(E)\in H^2(X)$.


The $H$-functions of $(Y_+,D_{Y,+})$ and $(Y_-,D_{Y,-})$ are easily identified following directly from the expression of the $H$-functions. 
\begin{theorem}\label{thm-toric-ci}
For any $d_+\in \mathbb K_+$, $d_-\in \mathbb K_-$ such that $d_+-d_-\in \mathbb Q e$, let $f_\pm=[d_\pm]$ and $\vec d_\pm=(D_i\cdot d_\pm)_{i\in I_\pm}$. We have
\[
H_{(Y_+,D_{Y,+}),f_+,\vec d_+}=H_{(Y_-,D_{Y,-}),f_-,\vec d_-} 
\]
with a change of variable (\ref{change-of-variables}).
\end{theorem}

Note that Theorem \ref{thm-toric-ci} was written for the non-extended relative $I$-function, but it is clear that it holds for extended relative $I$-function as well similar to the discussion in Section \ref{sec:extended-I-function}.

\subsection{Exchanging the role of divisors}\label{sec:exchange-divisors}
Given a hypersurface $D$ in a smooth projective variety/ Deligne--Mumford stack $X$. There are several different Gromov--Witten theories that one can consider. One can consider the absolute Gromov--Witten theory of $D$. One can also consider relative Gromov--Witten theory of $(X,D)$. \footnote{If $D$ is nef, we can also consider the local Gromov--Witten theory of $\mathcal O_X(-D)$, which, in genus zero, can be considered as a subset of relative Gromov--Witten theory of $(X,D)$ (\cite{vGGR}, \cite{TY20b}, \cite{BNTY}).}

In Section \ref{sec:complete-intersection-1}, for each line bundle $L_{i,\pm}$, we considered the Gromov--Witten theory of the complete intersections. Here, we can change the role of some of the divisors. In other words, instead of considering all of them as complete intersection divisors, we consider some of them as relative divisor and consider the corresponding relative Gromov--Witten theory.

\subsubsection{Motivation}

Under this setting, there may not be a birational transformation between the targets. In particular, they may be of different dimensions. However, the Gromov--Witten invariants are still related. This indeed has a motivation from mirror symmetry. We recall the following.
\begin{defn}[\cite{DHT}, Definition 2.1]\label{def-LG}
A Landau-Ginzburg model of a quasi-Fano variety $X$ is a pair $(X^\vee,W)$ consisting of a K\"ahler manifold $X^\vee$ satisfying $h^1(X^\vee)=0$ and a proper map $W:X^\vee\rightarrow \mathbb C$, where $W$ is called the superpotential.
\end{defn}
Furthermore, one expects that if $(X^\vee, W)$ is a Landau--Ginzburg model of $X$, then the generic fiber of $W$ should be mirror to the smooth anticanonical hypersurface in $X$. Therefore, one considers the Landau--Ginzburg model $(X^\vee, W)$ as a mirror of the smooth pair $(X,D)$.  More generally, if $(X,D)$ is a log Calabi-Yau pair with $D$ being simple normal crossings. The mirror of $(X,D)$, under some positivity assumption (e.g. Fano or quasi-Fano), is a higher rank Landau--Ginzburg model and the generic fibers of the Landau--Ginzburg model are Calabi-Yau varieties in lower dimensions. More precisely,

\begin{definition}[\cite{DKY21}, Definition 2.3]
A Landau--Ginzburg model of rank $1$ is the ordinary Landau--Ginzburg model in Definition \ref{def-LG}. For $n\geq 2$, a Landau--Ginzburg model of rank $n$ is a pair $(X^\vee, h)$, where
\[
h:=(h_1,\ldots,h_n): X^\vee\rightarrow \mathbb C^n,
\]
such that 
\begin{itemize}
    \item the generic fiber of $h_i$ together with the restriction of 
    \[
    \hat{h}_i:=(h_1,\ldots,h_{i-1},h_{i+1},\ldots,h_n)
    \]
    to this fiber is a Landau--Ginzburg model of rank $(n-1)$.
    \item by composing with the summation map $\Sigma:\mathbb C^n\rightarrow \mathbb C$, we get a non-proper ordinary Landau--Ginzburg model
    \[
    W:=\Sigma\circ h: X^\vee\rightarrow \mathbb C.
    \]
\end{itemize}
\end{definition}

If we assume that $D=D_1+D_2$ is a simple normal crossings anticanonical divisor of $X$ with two smooth irreducible components $D_1$ and $D_2$, and the intersection $D_{12}:=D_1\cap D_2$ is smooth. Then the rank $2$ Landau--Ginzburg model for $(X,D)$ is defined as follows.

\begin{definition}[\cite{Lee21}, Definition 1.4.1]\label{defn:hybrid-LG}
A rank $2$ Landau--Ginzburg model of a quasi-Fano variety $X$ with a simple normal crossings anticanonical divisor $D=D_1+D_2$ is a pair
\[
(X^\vee,h=(h_1,h_2):X^\vee\rightarrow \mathbb C^2),
\]
where $X^\vee$ is a K\"ahler manifold 
such that
\begin{itemize}
    \item a generic fiber of $h_1$, denoted $D^\vee_1$, with 
    \[
    h|_{D^\vee_1}=h_2: D^\vee_1\rightarrow \mathbb C
    \]
    is a rank $1$ Landau--Ginzburg model mirror to $(D_1,D_{12})$;
   \item    a generic fiber of $h_2$, denoted  $D^\vee_2$, with 
    \[
    h|_{D^\vee_2}=h_1: D^\vee_2\rightarrow \mathbb C
    \]
    is a rank $1$ Landau--Ginzburg model mirror to $(D_2,D_{12})$;
    \item by composing with the summation map $\Sigma:\mathbb C^2\rightarrow \mathbb C$, we get a non-proper ordinary Landau--Ginzburg model
    \[
    W:=\Sigma\circ h: X^\vee\rightarrow \mathbb C.
    \]
\end{itemize}
\end{definition}
Note that, the generic fiber of the multi-potential $h$ is mirror to the codimension $2$ Calabi--Yau $D_{12}$.

Therefore, we can consider Landau--Ginzburg models, mirror to $(X,D_1+\cdots +D_n)$, as families of Calabi--Yau varieties which are mirror to the Calabi--Yau complete intersection $\cap_{i=1}^n D_i$. In other words, the mirror for $(X,D_1+\cdots +D_n)$ and the mirror for $\cap_{i=1}^n D_i$ are two families of Calabi--Yau which are mirror to $\cap_{i=1}^n D_i$. One can compute relative periods of these Landau--Ginzburg models as in \cite{DKY21}. These relative periods are computed from the periods of the mirror of the Calabi--Yau complete intersection via pullback.

\begin{remark}
Although the motivation from mirror symmetry is in the setting of log Calabi--Yau pairs with some positivity assumptions on the ambient space and many interesting examples are within this setting, our general discussion will not require the Calabi--Yau condition or the positivity assumption on the ambient space.  
\end{remark}

\subsubsection{Examples}

Our set-up includes many examples that do not appear in \cite{CIJ} and \cite{AS}.
\begin{example}\label{ex-exchange-divisor}
   We consider $\mathbb P^3$ and a smooth quartic threefold $Q_4$ relative to their smooth anticanonical divisors respectively. Since $\mathbb P^3$ and the quartic threefold are both hypersurfaces of the same ambient toric variety $\mathbb P^4$, there is no need to perform wall-crossing. Then we will see that the ambient parts of their relative quantum cohomology are identical. 
   
   This can also be seen from their mirrors. The generic fibers of their mirror Landau--Ginzburg models are both mirror to quartic $K3$ surfaces. Therefore, we have two families of $M_2$-polarized $K3$ surfaces. Their relative periods have been computed in \cite{DKY}. On the other hand, in the Fanosearch program (see, for example, \cite{CCGGK}), one considers a Laurent polynomial as the mirror to a Fano variety. As mentioned in \cite{DKY}, the classical periods computed from the Laurent polynomials agree with the relative periods in \cite{DKY}. This is another evidence that one should consider a Landau--Ginzburg model as a mirror to a smooth log Calabi--Yau pair. The classical periods (and relative periods) of $(\mathbb P^3,K3)$ and $(Q_4,K3)$ are:
    \[
 f_{(\mathbb P^3,K3)}(t)=\sum_{d\geq 0}\frac{(4d)!}{(d!)^4}t^{4d}, \quad \text{and} \quad  f_{(Q_4,K3)}(t)=\sum_{d\geq 0}\frac{(4d)!}{(d!)^4}t^{d}.
 \]
They are the same up to a simple change of variables: $t^4\mapsto t$. This is actually because of the difference between the generalized functional invariants of the two families of $M_2$-polarized $K3$ surfaces which are Landau--Ginzburg mirrors of $(\mathbb P^3,K3)$ and $(Q_4,K3)$. 

The $H$-functions for these two smooth pairs are as follows
\[
H_{(\mathbb P^3,K3)}(y)=e^{\frac{H\log y}{2\pi i}}\sum_{d\geq 0} y^d \frac{\Gamma(1+\frac{4H}{2\pi i}+4d)}{\Gamma (1+\frac{H}{2\pi i}+d)^4}[\textbf{1}]_{4d},
\]
and
\[
H_{(Q_4,K3)}(y)=e^{\frac{H\log y}{2\pi i}}\sum_{d\geq 0} y^d \frac{\Gamma(1+\frac{4H}{2\pi i}+4d)}{\Gamma (1+\frac{H}{2\pi i}+d)^4}[\textbf{1}]_{d}.
\]
Therefore, we have
\[
H_{(\mathbb P^3,K3),4d}(y)=H_{(Q_4,K3),d}(y).
\]
 \end{example}
 
  In general, given a log Calabi--Yau pair $(X,D)$, periods of a Calabi-Yau manifold $D$ and the relative periods of the pair $(X,D)$ are not the same. But they are related by a change of variable, called the generalized functional invariant map in \cite{DKY}. 
  \begin{example}
  Let $X$ be a complete intersection of bidegrees $(1,1)$ and $(4,0)$ in $\mathbb P^4\times \mathbb P^1$. Its smooth anticanonical hypersurface, whose bidegree is $(0,1)$, is a quartic $K3$ surface in $\mathbb P^3$. On the other hand, we can consider a complete intersection of bidegrees $(1,1)$ and $(0,1)$ in $\mathbb P^4\times \mathbb P^1$, which is just $\mathbb P^3$. Its smooth anticanonical hypersurface is a quartic $K3$ surface. Here, we exchange the roles of the divisors $(4,0)$ and $(0,1)$.
  
  In \cite{DKY}*{Section 3.2.3}, the relative period for $(X,K3)$ is defined as a pullback of the period of the quartic $K3$ via the generalized functional invariant:
  \[
  f_{(X,K3)}(y_1,y_2)=\sum_{d_1,d_2\geq 0}\frac{(4d_1)!(d_1+d_2)!}{(d_1!)^5(d_2!)}y_1^{d_1}y_2^{d_2}.
  \]
  Hence, we have
  \[
  f_{(X,K3)}(t^4,0)=f_{(\mathbb P^3,K3)}(t).
  \]
  In other words, they are related by a change of variable $t\rightarrow t^4$ and setting $y_2=0$. The $H$-function for $(X,K3)$ is
  \[
H_{(X,K3)}(y_1,y_2)=e^{\frac{H\log y_1+P\log y_2}{2\pi i}}\sum_{d_1,d_2\geq 0} y^d \frac{\Gamma(1+\frac{4H}{2\pi i}+4d_1)\Gamma(1+\frac{H+P}{2\pi i}+d_1+d_2)}{\Gamma (1+\frac{H}{2\pi i}+d_1)^5\Gamma(1+\frac{P}{2\pi i}+d_2)}[\textbf{1}]_{(d_1,0)},
\]
 We have
 \[
i^*H_{(X,K3),(d,0)}(y,0)=H_{(\mathbb P^3,K3),4d}(y), 
 \]
 where $i:\mathbb P^3\subset \mathbb P^4 \hookrightarrow \mathbb P^4\times \mathbb P^1$.
 

  Note that $X$ is actually a blow-up of a smooth quartic threefold $Q_4$ along a complete intersection of two degree one hypersurfaces in $Q_4$. We also have a relation
  \[
  H_{(X,K3),(d,0)}(y,0)=H_{(Q_4,K3),d}(y).
  \]
  Hence, this example also gives a blow-up formula for relative quantum cohomology. But this is different from Section \ref{sec:toric-pairs} since here we choose smooth divisors, instead of simple normal crossings divisors. In Section \ref{sec:smooth-divisor}, we will study how relative quantum cohomology changes when we choose divisors to be smooth instead of simple normal crossings.
  \end{example}

  \subsubsection{General case}
  
Now we consider toric complete intersections in general. We will use the set-up in Section \ref{sec:complete-intersection-1}. For simplicity, we consider the case when $k=2$. In other words, we have line bundles $L_{1,\pm}$ and $L_{2,\pm}$ over $X_\pm$. Let $s_{i,\pm}$ be regular sections of $L_{i,\pm}\rightarrow X_\pm$ for $i=1,2$. Let $Z_+\subset X_+$ be the hypersurface defined by $s_{1,+}$ and $Z_-\subset X_-$ be the hypersurface defined by $s_{2,-}$. Let $D_{2,+}$ be the smooth divisor of $X_+$ defined by $s_{2,+}$ and $D_{1,-}$ be the smooth divisor of $X_-$ defined by $s_{1,-}$. Let 
\[
D_{Z,+}=(D_{2,+}+D_+)\cap Z_+=D_{2,+}\cap Z_++\sum_{i\in \{M_+\cup M_-\}\setminus S_+ }\bar D_i\cap Z_+
\]
and 
\[
D_{Z,-}=(D_{1,-}+D_-)\cap Z_-=D_{1,-}\cap Z_-+\sum_{i\in \{M_+\cup M_-\}\setminus S_- }\bar D_i\cap Z_-.
\]
Then we can compare the ambient parts of the relative quantum cohomology of $(Z_+,D_{Z,+})$ and $(Z_-,D_{Z,-})$ pulled back from the corresponding relative quantum cohomology of the ambient toric varieties $X_\pm$. The relative $I$-functions for $(Z_+,D_{Z,+})$ and $(Z_-,D_{Z,-})$ are as follows:
\begin{align*}
I_{(Z_+,D_{Z,+})}(y,z)=ze^{t_+/z}\sum_{d\in\mathbb K_{+}}y^{d}\left(\prod_{i\in M_0 }\frac{\prod_{a\leq 0,\langle a\rangle=\langle D_i\cdot d\rangle}(\bar{D}_i+az)}{\prod_{a \leq D_i\cdot d,\langle a\rangle=\langle D_i\cdot d\rangle}(\bar{D}_i+az)}\right)\\
\quad \left(\prod_{0<a\leq v_1\cdot d}(v_1+az)\right)\textbf{1}_{[-d]}I_{D_{Z,+},d},
\end{align*}
and 
\begin{align*}
I_{(Z_-,D_{Z,-})}(y,z)=ze^{t_-/z}\sum_{d\in\mathbb K_{-}}\tilde y^{d}\left(\prod_{i\in M_0}\frac{\prod_{a\leq 0,\langle a\rangle=\langle D_i\cdot d\rangle}(\bar{D}_i+az)}{\prod_{a \leq D_i\cdot d,\langle a\rangle=\langle D_i\cdot d\rangle}(\bar{D}_i+az)}\right)\\
\quad \left(\prod_{0<a\leq v_2\cdot d}(v_2+az)\right)\textbf{1}_{[-d]}I_{D_{Z,-},d},
\end{align*}
where
\[
I_{D_{Z,+},d}=\frac{\prod_{0<a< v_2\cdot d}(v_2+az)}{\prod_{i\in I_+, D_i\cdot d>0}(\bar D_i+(D_i\cdot d)z)}[\textbf{1}]_{(-D_i\cdot d)_{i\in I_+},v_2\cdot d},
\]
and
\[
I_{D_{Z,-},d}=\frac{\prod_{0<a<v_1\cdot d}(v_1+az)}{\prod_{i\in I_-, D_i\cdot d>0}(\bar D_i+(D_i\cdot d)z)}[\textbf{1}]_{(-D_i\cdot d)_{i\in I_-},v_1\cdot d}.
\]

The $H$-functions are
\begin{align}
H_{(Z_+,D_{Z,+})}(y)=e^{\frac{t_+}{2\pi i}}\sum_{d\in\mathbb K_{+}}y^d\left(\frac{
\prod_{i=1}^2 \Gamma(1+\frac{v_i}{z}+v_i\cdot d)
}{\prod_{i\in M_0}\Gamma (1+\frac{\bar D_i}{z}+D_i\cdot d)}\right)\textbf{1}_{[d]}[\textbf{1}]_{(D_i\cdot d)_{i\in I_+},v_2\cdot d}.
\end{align}
and
\begin{align}
H_{(Z_-,D_{Z,-})}(y)=e^{\frac{t_-}{2\pi i}}\sum_{d\in\mathbb K_{-}}y^d\left(\frac{
\prod_{i=1}^2 \Gamma(1+\frac{v_i}{z}+v_i\cdot d)
}{\prod_{i\in M_0}\Gamma (1+\frac{\bar D_i}{z}+D_i\cdot d)}\right)\textbf{1}_{[d]}[\textbf{1}]_{(D_i\cdot d)_{i\in I_-},v_1\cdot d}.
\end{align}

The $H$-functions of $(Z_+,D_{Z,+})$ and $(Z_-,D_{Z,-})$ are easily identified following directly from the expression of the $H$-functions. 
\begin{theorem}\label{thm-exchange-divisor}
For any $d_+\in \mathbb K_+$, $d_-\in \mathbb K_-$, such that $d_+-d_-\in \mathbb Q e$, let $f_\pm=[d_\pm]$, $\vec d_+=((D_i\cdot d_+)_{i\in I_+},v_2\cdot d_+)$ and $\vec d_-=((D_i\cdot d_-)_{i\in I_-},v_1\cdot d_-)$. We have
\[
H_{(Z_+,D_{Z,+}),f_+,\vec d_+}=H_{(Z_-,D_{Z,-}),f_-,\vec d_-} 
\]
with a change of variable (\ref{change-of-variables}).
\end{theorem}

\section{Relative quantum cohomology with non-toric divisors}\label{sec:smooth-divisor}

In Section \ref{sec:toric-pairs}, the irreducible components of $D_+$ and $D_-$ are chosen to be toric invariant divisors. In this section, we consider the case when the irreducible components of $D_+$ and $D_-$ are not necessarily toric invariant divisors. Without loss of generality, we assume that $D_+$ and $D_-$ are smooth divisors. Therefore, we indeed consider the relative Gromov--Witten theory of \cite{FWY}.

We use the toric set-up in Section \ref{sec:toric-birational}. Let $D_\pm=\sum_{i=0}^m a_{i,\pm} D_i$.
Note that
\begin{align*}
    f_+^*(D_+)=\sum_{i=1}^m a_{i,+}\tilde{D}_i+\left(\sum_{i=1}^m a_{i,+}D_i\cdot e\right)E,
\end{align*}
and
\begin{align*}
    f_-^*(D_-)=\sum_{i=1}^m a_{i,-}\tilde{D}_i-\left(\sum_{i=1}^m a_{i,-}D_i\cdot e\right)E.
\end{align*}
Recall that 
\[
K_{\tilde X}=f_-^*K_{X_-}+\left(1+\sum_{i\in M_-} D_i\cdot e\right)E=f_+^*K_{X_+}+\left(1-\sum_{i\in M_+} D_i\cdot e\right)E.
\]
and
\[
f_+^*K_{X_+}=f_-^*K_{X_-}+\left(\sum_{i=1}^m D_i\cdot e\right)E.
\]
We would like to choose $a_{i,\pm}$ such that
\[
f_+^*(K_{X_+}+D_+)=f_-^*(K_{X_-}+D_-).
\]
For example we can set
\[   
a_{i,\pm}=\left\{
\begin{array}{ll}
      1 & i\in (M_+\cup M_-) \setminus S_\pm\\
      0 & \text{otherwise}.
\end{array} 
\right. \]
Then we have $D_+=\sum_{i\in (M_+\cup M_-)\setminus S_+} D_i$ and $D_-=\sum_{i\in (M_+\cup M_-)\setminus S_-} D_i$.
We also assume that $D_+$ and $D_-$ are nef. 

\begin{remark}
If $D_+$ and $D_-$ are not nef, we may also allow some $a_{i,+}=a_{i,-}\neq 0$ for $i\in M_0$ such that $D_+$ and $D_-$ are nef. The results in this section will not be affected. For simplicity, we choose $D_+=\sum_{i\in (M_+\cup M_-)\setminus S_+} D_i$ and $D_-=\sum_{i\in (M_+\cup M_-)\setminus S_-} D_i$ and assume that $D_+$ and $D_-$ are smooth and nef. The proof works for more general choice of $D_\pm$. 
\end{remark}

\subsection{Via local-orbifold correspondence}\label{sec:local-orbifold}

In this section, we specialize to the genus zero invariants with maximal contact orders. According to the local-relative/orbifold correspondence of \cite{vGGR}, \cite{TY20b} and \cite{BNTY}, genus zero relative Gromov--Witten invariants of $(X,D)$(or the formal Gromov--Witten invariants of infinite root stacks for simple normal crossings pairs) with maximal contact orders coincide with genus zero local Gromov--Witten invariants of $\mathcal O_X(-D)$ (or the direct sum of line bundles $\mathcal O_X(-D_i)$ for simple normal crossings pairs). Recall that we consider the case when $D_+$ and $D_-$ are smooth. The case when they are simple normal crossings works similarly. 

We compare the total space of $T_+:=\mathcal O_{X_+}(-D_+)$ and $T_-:=\mathcal O_{X_-}(-D_-)$. Note that, what we consider here is slightly more general than what was considered in \cite{MS}, where they consider the case of blow-ups. In here, we include the cases when  $\phi: \xymatrix{
X_+\ar@{-->}[r]& X_-}$ is a discrepant resolution in general (including a weighted blow-up). We will also consider the case of toric complete intersections.

Given a divisor $D=\sum_{i=1}^m a_iD_i$ with $a_i\in \mathbb Z$. The support function $\phi_D$ for $D$ is
\[
\phi_D: N\otimes \mathbb R\rightarrow \mathbb R
\]
a piecewise-linear function characterized by the condition: $\phi_D(\bar{b}_i)=-a_i$ for $i\not\in S$. We impose the following assumptions:
\begin{assumption}\label{assum-convex}
The function satisfies the following:
\begin{itemize}
    \item For each $\sigma\in \Sigma_\omega$, there exists an element $m_\sigma\in N^\vee$ such that
    \[
    \phi_D(n)=\langle m_\sigma,n\rangle,
    \]
    for $n\in |\sigma|$;
    \item The graph of $\phi_D$ is convex and $\phi_D(\bar{b}_i)\geq a_i$ for $i\in S$.
\end{itemize}
\end{assumption}

By \cite{MS}*{Lemma 3.8}, Assumption \ref{assum-convex} implies that $D$ is basepoint free and $\mathcal O_{X}(D)$ is a convex line bundle on $X$. We assume that $D_+$ and $D_-$ satisfy Assumption \ref{assum-convex}.

Both $T_+$ and $T_-$ can be realized as GIT quotients. We define
\[
\hat{\beta}:\mathbb Z^{m+1}\rightarrow N\oplus \mathbb Z
\]
by
\[   
\hat{\beta}=\left\{
\begin{array}{ll}
      (\beta(e_i),1) & i\in\{1,\ldots,m\} \text{ and }i\in M_+\cup M_- \\
      (\beta(e_i),0) & i\in\{1,\ldots,m\} \text{ and } i\in M_0 \\
      (0, 1) & \text{If $i=m+1$}.
\end{array} 
\right. \]
We consider the fan sequence:
\begin{equation}\label{fan-seq-local}
0 \longrightarrow \mathbb{L}:=\text{ker}(\hat{\beta}) \longrightarrow \mathbb{Z}^{m+1} \stackrel{\hat{\beta}}{\longrightarrow} N\oplus \mathbb Z\longrightarrow 0,
\end{equation}
where the map $\mathbb L\rightarrow \mathbb Z^{m+1}$ is given by $(D_1,\ldots,D_m,-D)$ and $D=\sum_{i\in M_+\cup M_-} D_i$.

Then we consider the wall-crossing. Recall that, there are three types of toric wall-crossing as mentioned in Section \ref{sec:wall-crossing}: Type I is flip and flop; Type II is discrepant/crepant resolution; Type III is gerbe constructions and root constructions. Type III is relatively well-understood in the literature, so we do not plan to study it here. For Type I and Type II, without loss of generality, we assume that $\sum_{i=1}^m D_i\cdot e\geq 0$. In the case of flip, we have $\sum_{i=1}^m D_i\cdot e> 0$ and $\sum_{i=1}^m \bar{D}_i$ is not nef in $X_-$. So we will not study flip here. Therefore, we will focus on the case when $\phi$ is a discrepant resolution, although we will mention the crepant case as well.

Let $\hat\Sigma_{\omega_\pm}$ be the fans corresponding to the GIT quotients $[\mathbb C^{m+1}\sslash_{\omega_\pm}K]$, where the action of $K$ on $\mathbb Z^{m+1}$ is given by $(D_1,\ldots,D_m,-D)$. Define $\hat{b}_i:=\hat{\beta}(e_i)$.  The following proposition follows from \cite{MS}*{Proposition 5.1}
\begin{prop}\label{prop-T-+}
The total space $T_+$ is the toric stack which can be expressed as the GIT quotient $[\mathbb C^{m+1}\sslash_{\omega_+}K]$.
\end{prop}
\begin{proof}
Given a cone $\sigma_I\in \Sigma_{+}$, let $\hat\sigma_I$ be the cone generated by $\overline{\hat{b}_i}$ for $i\in I\cup \{m+1\}$. By the definition of the anticones, $\sigma_I\in \Sigma_{+}$ implies  $\hat\sigma_I\in \hat\Sigma_{+}$. Following the proof of \cite{MS}*{Proposition 5.1}, the convexity assumption (Assumption \ref{assum-convex}) implies that $\{\hat{\sigma}\}_{\sigma\in \Sigma_{+}}$ contains all the cones of $\Sigma_{+}$. Then it is straightforward that the GIT quotient $[\mathbb C^{m+1}\sslash_{\omega_+}K]$ is the total space $T_+$.
\end{proof}

Now we would like to see what $[\mathbb C^{m+1}\sslash_{\omega_-}K]$ is. 

\begin{prop}\label{prop-wall-crossing-local}
If $\phi: \xymatrix{
X_+\ar@{-->}[r]& X_-}$ is a crepant transformation, then the induced birational morphism $$\hat{\phi}: \xymatrix{
[\mathbb C^{m+1}\sslash_{\omega_+}K]\ar@{-->}[r]& [\mathbb C^{m+1}\sslash_{\omega_-}K]}$$ is a crepant transformation of the same type. If $\phi$ is a discrepant resolution, the morphism $$\hat{\phi}: \xymatrix{
[\mathbb C^{m+1}\sslash_{\omega_+}K]\ar@{-->}[r]& [\mathbb C^{m+1}\sslash_{\omega_-}K]}$$ is a flop.
\end{prop}
\begin{proof}
Since we choose $D_{m+1}:=D=\sum_{i\in M_+\cup M_-} D_i$, it is clear that the morphism $\hat{\phi}$ is crepant. Recall that we assume $\sum_{i=1}^m D_i\cdot e\geq 0$. Then $m+1\in \hat{M}_{0}$ if $\sum_{i=1}^m D_i\cdot e= 0$ and $m+1\in \hat{M}_{-}$ if $\sum_{i=1}^m D_i\cdot e> 0$. Therefore, if $\phi: \xymatrix{
X_+\ar@{-->}[r]& X_-}$ is a crepant transformation, then $\hat{\phi}$ is also a crepant transformation of the same type. When $\phi: \xymatrix{
X_+\ar@{-->}[r]& X_-}$ is a discrepant resolution, we have $\#(\hat M_-)=2$ and $\#(\hat M_+)\geq 2$. Therefore, $\hat{\phi}$ is a flop.
\end{proof}

Then we would like to understand the relation between $[\mathbb C^{m+1}\sslash_{\omega_-}K]$ and $T_-$.
\begin{prop}\label{prop-cone-T-}
For each cone $\sigma_I\in \Sigma_{-}$, the cone $\hat{\sigma}_I$, which is generated by $\overline{\hat{b}_i}$ for $i\in I\cup \{m+1\}$, is in $\hat{\Sigma}_{-}$.
   \end{prop}

\begin{proof}
This follows from an identical argument as in Proposition \ref{prop-T-+}.
\end{proof}

The fan formed by the cones $\hat{\sigma}_I$ in Proposition \ref{prop-cone-T-} gives the toric stack $T_-$. 
In the crepant case, these are all the cones of $\hat{\Sigma}_{-}$ and we simply have
\begin{proposition}
For crepant toric wall-crossings $\phi: \xymatrix{
X_+\ar@{-->}[r]& X_-}$, we have $[\mathbb C^{m+1}\sslash_{\omega_-}K]=T_-$.
\end{proposition}
Then we have a theorem that is a direct consequence of the crepant transformation conjecture for toric complete intersections proved in \cite{CIJ}.
\begin{theorem}
The narrow relative quantum D-module (or the $I$-function) of $(X_+,D_+)$ can be analytic continued to the narrow relative quantum D-module (or the $I$-function) of $(X_-,D_-)$.
\end{theorem}

In the discrepant case, we have some extra cones.
\begin{proposition}\label{prop-compactification}
For Type II discrepant toric wall-crossings $\phi: \xymatrix{
X_+\ar@{-->}[r]& X_-}$, we have the following. For every cone $\sigma_I$ of $\Sigma_{-}$ such that $I$ contains $M_+$, the cone generated by $I\cup M_-$ is in $\hat{\Sigma}_{-}$. And, $[\mathbb C^{m+1}\sslash_{\omega_-}K]$ is a partial compactification of $T_-$, denoted by $\overline{T}_-$.
\end{proposition}
\begin{proof}
The fan $\Sigma_-$ does not contain cones of the form $\sigma_J$, where $J$ contains $M_-$. Otherwise, there exists constants $c_j>0$ such that $\omega_-=\sum_{j\in J^c\subset M_0\cup M_+}c_jD_j$, then $\omega_-\cdot e\geq 0$ which contradicts the choice of $\omega_-$. 

According Remark \ref{rmk-circuit}, the toric wall-crossing can be described as a modification of the circuit $\hat{M}_+\cup \hat{M}_-$. Note that, for discrepant wall-crossing, we have $\hat{M}_-=M_-\cup\{m+1\}$ and $\hat{M}_+=M_+$. The difference between the modification of the circuits $\hat{M}_+\cup\hat{M}_-$ and $M_+\cup M_-$ is the following: for the cone $\sigma_I$ of $\Sigma_{-}$ such that $I$ contains $M_+$, the cone $I\cup M_-$ is also a cone of $\hat{\Sigma}_{-}$. These are the extra cones of $\hat{\Sigma}_{-}$ that serve as a partial compactification of $T_-$.

\end{proof}

By Proposition \ref{prop-compactification}, $\overline{T}_-$ is a partial compactification of $T_-$, then we need to relate the $I$-functions of $\overline{T}_-$ and $T_-$. Similar to \cite{MS}*{Lemma 5.9}, the following relation between $\overline{T}_-$ and $T_-$ holds.
\begin{lemma}
\begin{enumerate}
    \item The map 
    \begin{align*}
    \mathbb L\otimes \mathbb R&\rightarrow (\mathbb L\otimes \mathbb R)\oplus \mathbb R\\
    l&\mapsto (l,0)
    \end{align*}
    defines a canonical isomorphism $H_2(\overline{T}_-;\mathbb R)\cong H_2(T_-;\mathbb R)\oplus \mathbb R$.
    \item The map 
    \begin{align*}
    \mathbb L^\vee\otimes \mathbb R&\rightarrow (\mathbb L^\vee\otimes \mathbb R)\oplus \mathbb R\\
    l&\mapsto (l,0)
    \end{align*}
    defines a canonical isomorphism $H^2(\overline{T}_-;\mathbb R)\cong H^2(T_-;\mathbb R)\oplus \mathbb R$.
    \item The Mori cones are related as
    \[
    \on{NE}(\overline{T}_-)=\on{NE}(T_-)\times \mathbb R_{\geq 0} (0,-1).
    \]
\end{enumerate}
\end{lemma}
Therefore, we can write
\[
\hat{d}=d_-+d^\prime\in  \on{NE}(\overline{T}_-)=\on{NE}(T_-)\times \mathbb R_{\geq 0} (0,-1),
\]
and
\[
\hat{y}^{\hat{d}}=y^{d_-}\cdot (y^\prime)^{d^\prime}.
\]

A mild generalization of \cite{MS}*{Theorem 1.3} gives the following.
\begin{theorem}\label{thm-specialization}
The monodromy invariant part of the narrow quantum D-module of $\overline{T}_-$ around $y^\prime=0$, when restricted to $y^\prime=0$, contains a submodule which maps surjectively to the narrow quantum D-module of $T_-$ via $i^*$, where $i: T_-\hookrightarrow \overline{T}_-$ is the inclusion map.
\end{theorem}

In terms of $I$-functions, the $I$-function (or the $H$-function) of $T_+$ can be analytically continued to the $I$-function (or the $H$-function) of $\overline{T}_-$ which is identified with the $I$-function (or the $H$-function) of $T_-$ after specializing to $y^\prime=0$.

By the local-relative/orbifold correspondence of \cite{vGGR}, \cite{TY20b} and \cite{BNTY}, we have
\begin{theorem}\label{thm-narrow-qd}
After analytic continuation, the monodromy invariant part of the narrow quantum D-module of $(X_+,D_+)$ around $y^\prime=0$, when restricted to $y^\prime=0$, contains  the narrow quantum D-module of $(X_-,D_-)$ as a subquotient.
\end{theorem}

\begin{remark}
When the divisors $D_+$ and $D_-$ are simple normal crossings and each irreducible component is nef, then the genus zero formal Gromov--Witten invariants of infinite root stacks $X_{\pm, {D_\pm,\infty}}$ with maximal contact orders coincide with the local Gromov--Witten invariants of the total space of the corresponding vector bundles $T_\pm$. Then $T_+$ and $\overline{T}_-$ are still toric stacks related by crepant wall-crossings. Theorem \ref{thm-narrow-qd} holds in this case too.
\end{remark}

\begin{remark}
Given a discrepant toric wall-crossing as in Section \ref{sec:wall-crossing}, \cite{AS} studied how absolute Gromov--Witten invariants are related. It was shown in \cite{AS} that the Laplace transform of the regularized $I$-function of $X_-$ is related to the $I$-function of $X_+$ via a linear transformation. As far as we know, a direct geometric meaning of the regularized $I$-function has not been explored.  
In our set-up, we consider relative $I$-functions instead of absolute $I$-functions and then study the analytic continuation. The relative $I$-function here is different from the regularized $I$-function considered in \cite{AS}.  
\end{remark}

\subsection{Via the hypersurface construction}\label{sec:hypersurface}
Let $X$ be a smooth projective variety and $D$ be a smooth nef divisor. Recall that, following \cite{FTY}, the root stack $X_{D,r}$ can be constructed as a hypersurface of a $\mathbb P^1[r]$-bundle $Y_{X_\infty,r}$ over $X$. In this section, we use this hypersurface construction to study wall-crossing behaviors of relative quantum cohomology of $(X,D)$. The advantage of this construction is that we can consider invariants that are not of maximal contact order, by considering the extended $I$-functions for the toric stack $Y_{\pm,X_\infty,r}$.

Recall that, we have the following hypersurface construction.
Consider the line bundle 
\[
L:=\cO_X(-D),
\] 
and the projectivization 
\[Y:=\mathbb P(L\oplus\cO_X) \xrightarrow{\pi} X.
\]
Let 
\[
X_\infty= \mathbb P(L)\subset \mathbb P(L\oplus\cO_X).
\]
Then $Y$ is the compactification of the total space of $L$ and $X_\infty$ is the divisor that we add in at infinity to compactify $L$. We have the following diagram from \cite{FTY}*{Section 3.1}.

\begin{center}
\begin{tikzcd}
& p^*(\cO_Y(1)) \arrow[r] \arrow[d]&\cO_Y(1) \arrow[d] \\
p^*\tilde{f}^{-1}(0)=X_{D,r} \arrow[r, hook, "i"]
& Y_{X_\infty,r} \arrow[d] \arrow[r, "p"] \arrow[u, bend left, "p^*\tilde{f}"] & Y\arrow[d,"\pi"] \arrow[u, bend left, "\tilde{f}"]\\
& X & X
\end{tikzcd}.
\end{center}

Given a toric birational map $\phi: \xymatrix{
X_+\ar@{-->}[r]& X_-}$ as in Section \ref{sec:wall-crossing}, we assume $D_+$ and $D_-$ are smooth, nef divisors that are linear equivalent to $\sum_{i\in (M_+\cup M_-)\setminus S_+} D_i$ and $\sum_{i\in (M_+\cup M_-)\setminus S_-} D_i$ respectively. Both $Y_+:=\mathbb P(\cO(-D_+)\oplus\cO_{X_+})$ and $Y_-:=\mathbb P(\cO(-D_-)\oplus\cO_{X_-})$ can be realized as GIT quotients. We define
\[
\tilde{\beta}:\mathbb Z^{m+2}\rightarrow N\oplus \mathbb Z
\]
by
\[   
\tilde{\beta}=\left\{
\begin{array}{ll}
      (\beta(e_i),1) & i\in\{1,\ldots,m\} \text{ and }i\in M_+\cup M_- \\
      (\beta(e_i),0) & i\in\{1,\ldots,m\} \text{ and } i\in M_0 \\
      (0, 1) & \text{If $i=m+1$}\\
      (0, -1) & \text{If $i=m+2$}. 
\end{array} 
\right. \]
We consider the fan sequence:
\begin{equation}\label{fan-seq-proj}
0 \longrightarrow \tilde{\mathbb{L}}:=\text{ker}(\tilde{\beta}) \longrightarrow \mathbb{Z}^{m+2} \stackrel{\tilde{\beta}}{\longrightarrow} N\oplus \mathbb Z\longrightarrow 0,
\end{equation}
where the map $\tilde{\mathbb L}=\mathbb L\oplus \mathbb Z\rightarrow \mathbb Z^{m+2}$ is given by $(\tilde{D}_1,\ldots,\tilde{D}_m,-\tilde{D}_{m+1},\tilde{D}_{m+2})$ with
\[
\tilde{D}_i=(D_i,0), \text{ for } i\in\{1,\ldots,m\}; 
\]
\[
\tilde{D}_{m+1}=(-D,1), \, \tilde{D}_{m+2}=(0,1), \text{ and } D=\sum_{i\in M_+\cup M_-} D_i.
\] 
Recall that $\omega_+$ and $\omega_-$ are stability conditions in chambers $C_+$ and $C_-$ that are separated by a hyperplane wall $W$ in $\mathbb L^\vee\otimes \mathbb R$.  We consider the stability conditions
\[
\tilde{\omega}_+=(\omega_+,\epsilon), \quad \text{and } \tilde{\omega}_-=(\omega_-,\epsilon),
\]
where $\epsilon$ is a very small positive real number. We have
\[
Y_+=X_{\tilde{\omega}_+}, \text{ and } Y_-=X_{\tilde{\omega}_-}.
\]
We also have
\[
[X_\infty]= \bar{\tilde D}_{m+2}.
\]

We assume $\sum_{i=1}^m D_i\cdot e\geq 0$. There are two cases
\begin{itemize}
    \item The crepant case: $\sum_{i=1}^m D_i\cdot e=0$.  Two stability conditions $\tilde{\omega}_+$ and $\tilde{\omega}_-$ are related by a single wall-crossing and the wall $\tilde W$ is span by two subspaces $(W,0)$ and $\mathbb R(0,1)$.
    \item The discrepant case: $\sum_{i=1}^m D_i\cdot e >0$. Two stability conditions $\tilde{\omega}_+$ and $\tilde{\omega}_-$ are related by crossing two walls in $\tilde{\mathbb L}^\vee\otimes \mathbb R$. The two walls are the hyperplanes
\[
\tilde{W}_1:=\on{span}\{(W,0),(0,1)\}, \quad \text{and } \tilde{W}_2:= \on{span}\{(W,0),(-D,1)\}.
\]
\end{itemize}
Note that $D=\sum_{i\in M_+\cup M_-} D_i$. If $\sum_{i=1}^m D_i\cdot e=0$, then $D$ is on the wall $W$ and $\tilde W_1=\tilde W_2=\tilde W$. Crossing the wall $W(=\tilde W_1)$ is crepant since $\sum_{i=1}^{m+2}\tilde D_i$ is on the wall $\tilde W_1$. But crossing the wall $\tilde W_2$ (if it does not coincide with $\tilde W_1$) is discrepant.

Similar to Section \ref{sec:local-orbifold}, we have
\begin{proposition}\label{prop-proj-bundle-I}
When $\phi: \xymatrix{
X_+\ar@{-->}[r]& X_-}$ is induced from toric crepant wall-crossing, then the induced birational transformation between $Y_+$ (resp. $Y_{+,X_\infty,r}$) and $Y_-$ (resp. $Y_{-,X_\infty,r}$) is again a crepant transformation of the same type.
\end{proposition}

Now we consider the case of discrepant resolution. 
\begin{proposition}\label{prop-proj-bundle-II}
If $\phi: \xymatrix{
X_+\ar@{-->}[r]& X_-}$ is a discrepant resolution induced from a Type II toric wall-crossing, then crossing the first wall $\tilde W_1$ induces a birational transformation between $Y_+$ (resp. $Y_{+,X_\infty,r}$) and a toric stack $\tilde Y_-$ (resp. $\tilde Y_{-,X_\infty,r}$). This birational transformation is a flop.

Crossing the wall $\tilde W_2$ gives a birational transformation between $\tilde Y_-$ (resp. $\tilde Y_{-,X_\infty,r}$)  and $Y_-$ (resp. $Y_{-,X_\infty,r}$). This birational transformation is a discrepant resolution.
\end{proposition}
\begin{proof}
The first wall-crossing is similar to the wall-crossing in Proposition \ref{prop-wall-crossing-local}. We have $m+1\in \tilde M_-$, then  $\#\{\tilde{M}_+\}\geq 2$ and $\#\{\tilde{M}_-\}\geq 2$. Since it is crepant, we conclude that it is a flop. 

For the second wall-crossing, $m+1\in \tilde M_0$ and $m+2\in \tilde M_+$. then  $\#\{\tilde{M}_+\}=\#\{M_+\}+1\geq 3$ and $\#\{\tilde{M}_-\}=\#\{M_-\}=1$. Therefore, it is a discrepant resolution.
\end{proof}

\begin{remark}\label{rmk-partial-compactification}
The geometric construction in this section also gives another explanation of the relation between $\overline{T}_-$ and $T_-$ in Section \ref{sec:local-orbifold}. Recall that $\overline{T}_-$ is a partial compactification of $T_-$. The first wall-crossing here is related to the wall-crossing in Section \ref{sec:local-orbifold} that gives a crepant transformation between $T_+$ and $\overline{T}_-$. The second wall-crossing here is related to the partial compactification as follows. Note that $\overline{T}_-=\tilde Y_-\setminus \bar{\tilde D}_{m+2}$ and $T_-=Y_-\setminus\bar{\tilde D}_{m+2}=Y_-\setminus X_\infty$. Moreover, the loci of indeterminacy are $\bar{\tilde D}_i\subset \tilde Y_-$ where $i\in \tilde{M}_-$ and $\bigcap_{i\in \tilde M_+}\bar{D}_i\subset X_\infty\subset Y_-$. By removing $X_\infty$, one removes the loci of indeterminacy in $Y_-$. Therefore, the complements are just related by partial compactification and the partial compactification is coming from the exceptional divisor in $\tilde Y_-$. The complement $\overline{T}_-\setminus T_-$ is codimension one.
\end{remark}

\if{
Write
\[
V_{+}:=\mathcal O_{Y_+}(-1), \text{ and } V_{-}:=\mathcal O_{Y_-}(-1).
\]
Their pullback to $Y_{\pm,X_\infty,r}$ are written as 
\[
V_{+,\bar{\tilde D}_{m+2},r}:=p^*_+(\mathcal O_{Y_+}(-1)), \text{ and } V_{-,r}:=p^*_-(\mathcal O_{Y_-}(-1)).
\]
Then Quantum Serre duality of \cite{Shoemaker18} implies that the ambient quantum D-module of $X_{\pm,D_\pm,r}$ is equivalent to the narrow quantum D-module of $V_{\pm,r}$. 
}\fi

The extended $I$-function of relative Gromov--Witten theory is given in \cite{FTY}*{Section 4.3}. It is taken as a limit of the extended $I$-function of root stacks which is obtained from the hypersurface construction as described in Section \ref{sec:hypersurface}. The extended $I$-function for a smooth pair $(X,D)$ is defined as follows. We take the extended data
\[
S:=\{a_1,\ldots,a_m\}.
\]
The $S$-extended $I$-function for $(X,D)$ is 
\[
I_{(X,D)}^{S}(Q,x,t,z)=I_{\on{pos}}+I_{\on{neg}},
\]
where
\begin{align*}
I_{\on{pos}}:=&ze^{t/z}\sum_{\substack{d\in\mathbb K,(k_1,\ldots,k_m)\in (\mathbb Z_{\geq 0})^m\\ \sum_{i=1}^mk_ia_i<D\cdot d} }y^{d}\left(\prod_{i=0}^{m}\frac{\prod_{a\leq 0,\langle a\rangle=\langle D_i\cdot d\rangle}(\bar{D}_i+az)}{\prod_{a \leq D_i\cdot d,\langle a\rangle=\langle D_i\cdot d\rangle}(\bar{D}_i+az)}\right)\frac{\prod_{i=1}^m x_i^{k_i}}{z^{\sum_{i=1}^m k_i}\prod_{i=1}^m(k_i!)}\\
&\frac{\prod_{0<a\leq D\cdot d}(D+az)}{D+(D\cdot d-\sum_{i=1}^mk_ia_i)z}\textbf{1}_{[-d]}[{\mathbf 1}]_{-D\cdot d+\sum_{i=1}^mk_ia_i},
\end{align*}
and 
\begin{align*}
I_{\on{neg}}:=&ze^{t/z}\sum_{\substack{d\in\mathbb K,(k_1,\ldots,k_m)\in (\mathbb Z_{\geq 0})^m\\ \sum_{i=1}^mk_ia_i\geq D\cdot d} }y^{d}\left(\prod_{i=0}^{m}\frac{\prod_{a\leq 0,\langle a\rangle=\langle D_i\cdot d\rangle}(\bar{D}_i+az)}{\prod_{a \leq D_i\cdot d,\langle a\rangle=\langle D_i\cdot d\rangle}(\bar{D}_i+az)}\right)\frac{\prod_{i=1}^m x_i^{k_i}}{z^{\sum_{i=1}^m k_i}\prod_{i=1}^m(k_i!)}\\
&\left(\prod_{0<a\leq D\cdot d}(D+az)\right)\textbf{1}_{[-d]}[{\mathbf 1}]_{-D\cdot d+\sum_{i=1}^mk_ia_i}.
\end{align*}

One can simplify the expression by considering the $H$-function. The $H$-function is
\begin{align}
H_{(X,D)}(y)=e^{\frac{t}{2\pi i}}\sum_{\substack{d\in\mathbb K\\ (k_1,\ldots,k_m)\in (\mathbb Z_{\geq 0})^m}}y^d\left(\frac{\Gamma (1+\frac{D}{2\pi i}+D\cdot d)}{\prod_{i=1}^m\Gamma (1+\frac{\bar D_i}{2\pi i}+D_i\cdot d)}\right)\textbf{1}_{[d]}[\textbf{1}]_{D\cdot d-\sum_{i=1}^mk_ia_i}\frac{\prod_{i=1}^m x_i^{k_i}}{\prod_{i=1}^m(k_i!)}.
\end{align}

We can state the relation between the quantum cohomology of $(X_+,D_+)$ and $(X_-,D_-)$ in terms of $H$-functions.

\begin{theorem}\label{thm-sm-extended}
The (extended) $H$-function of $(X_+,D_+)$ can be analytically continued to the (extended) $H$-function of $(X_-,D_-)$ under the specialization $y^\prime=0$. 
\end{theorem}

\begin{proof}

The proof consists of the following steps.

\begin{itemize}
    \item[Step 1] We consider the $I$-functions of $X_{D,r}$. By the hypersurface construction, We just need to study the $I$-function of $Y_{X_{\infty},r}$ and apply the orbifold quantum Lefschetz principle of \cite{Tseng}. By the orbifold quantum Serre duality of \cite{Tseng}, it is equivalent to considering the $I$-function of the total space of the line bundle $p^*\mathcal O_{Y}(-1)$. This is a toric stack and we will denote it by $L_r$. We use $\tilde D_{m+2}/r$ to denote the toric divisor with the $r$-th root construction. 
    
    \item[Step 2] Recall that $Y_+$ and $Y_-$ are related by two wall-crossings: the first wall-crossing is a flop and the second wall-crossing is a discrepant resolution. We consider how the toric stack $L_{+,r}$ changes under these two wall-crossings.
    The first wall-crossing relates the $H$-function of $L_{+,r}$ and the $H$-function of $\tilde L_{-,r}$ via analytic continuation in \cite{CIJ}*{Section 6.2}, where $\tilde L_{-,r}$ is the toric stack obtained from crossing the first wall. Note that the toric divisor $\tilde D_{m+2}/r$ is on the wall. The analytic continuation in \cite{CIJ}*{Section 6.2} essentially does not affect the factor involving $\tilde D_{m+2}/r$. Hence the analytic continuation is compatible with the limit $r\rightarrow \infty$. 
        \item[Step 3] The second wall-crossing is a flip (it ``becomes crepant" when $r\rightarrow \infty$). The toric stack obtained by crossing the second wall is denoted by $\overline{L}_{-,r}$. The $H$-function of $\tilde L_{-,r}$ is analytic. Note that $\tilde D_{m+2}/r$ is the $r$-th root of ${\tilde D}_{m+2}$ and it is on the positive side of the wall (i.e. $m+2\in \tilde M_{+}$). There is a factor of $\frac{1}{\Gamma(1+\bar{\tilde D}_{m+2}/r+(\tilde D_{m+2}\cdot d)/r)}$ in the $H$-function and it becomes $1/\Gamma(1)=1$ when $r\rightarrow \infty$. Follow the analytic continuation computation of \cite{CIJ}*{Section 6.2} carefully, the factor $\frac{1}{\Gamma(1+\bar{\tilde D}_{m+2}/r+(\tilde D_{m+2}\cdot d)/r)}$ will again be $1$ as $r\rightarrow \infty$ after analytic continuation. In other words, this analytic continuation is again compatible with the limit $r\rightarrow \infty$ and one obtains the analytic continuation between the $H$-function of $\tilde L_{-,\infty}$ and the $H$-function of $\overline{L}_{-,\infty}$.
        \item[Step 4] Proposition \ref{prop-compactification} can be applied here and implies that $\tilde L_{-,r}$ is a partial compactification of $L_{-,r}$. Then Theorem \ref{thm-specialization} implies that their $I$-functions are related by specialization of $y^\prime=0$. This specialization is also compatible with the limit $r\rightarrow \infty$. This concludes the proof. 
\end{itemize}

\end{proof}

\begin{remark}
The case of toric complete intersections works similarly, by combining the discussion in this section and the discussion in Section \ref{sec:complete-intersection-1}. So we will not repeat it here.
\end{remark}

\section{Connection to transitions}\label{sec:transition}

Let $\phi: \xymatrix{
X_+\ar@{-->}[r]& X_-}$ be a toric blow-up given by a Type II toric wall-crossing. Let $D_+$ and $D_-$ be smooth and given by
\[
D_+=\sum_{i\in (M_+\cup \{j\})} \bar D_i\subset X_+ \text{ and } D_-=\sum_{i\in M_+} \bar D_i\subset X_-. 
\]
When $D_+$ and $D_-$ are nef, Theorem \ref{thm-narrow-qd} states that the narrow relative quantum D-modules of $(X_+,D_+)$ and $(X_-,D_-)$ are related by analytic continuation and specialization. On the other hand, the divisors $D_+\subset X_+$ and $D_-\subset X_-$ are related by transitions \cite{MS}*{Proposition 4.5}. \cite{MS}*{Theorem 1.1} states that the ambient quantum D-modules of $D_+$ and $D_-$ are related by analytic continuation and specialization. Indeed, Theorem \ref{thm-narrow-qd} (in the case of toric blow-ups) follows from \cite{MS}*{Theorem 1.1} by the local-relative correspondence \cite{vGGR}, \cite{TY20b}, \cite{BNTY} and the quantum Serre duality. 

There is a rank reduction phenomenon caused by the specialization. We provide a partial answer to it. It is not surprising that the rank reduction has to do with the curves that are contracted by the map $\phi$. In this section, we will discuss the rank reduction phenomenon for transitions, the rank reduction phenomenon for the log-crepant transformation of smooth pairs works the same. 

In this section, we use slightly different notation to emphasize that we are considering blow-ups. We write 
\[
X:=X_-, \quad \tilde X:=X_+, \quad, D:=D_-, \quad \tilde D:=D_+. 
\]

We can simply consider the fan sequence (\ref{fan-seq-blow-up}) for the blow-up:
\begin{equation}
0 \longrightarrow \mathbb{L}\oplus \mathbb Z \longrightarrow \mathbb{Z}^{m+1} \stackrel{\tilde{\beta}}{\longrightarrow} N\longrightarrow 0,
\end{equation}
such that the map $\tilde{\beta}: \mathbb Z^{m+1}\rightarrow N$ is given by $\{\beta_1,\ldots,\beta_m,\beta_{m+1}\}$ and the map $\mathbb L\oplus \mathbb Z\rightarrow \mathbb Z^{m+1}$ is given by $(\tilde {D}_1,\ldots,\tilde {D}_m,\tilde{D}_{m+1})$ where
\[   
\tilde D_i=\left\{
\begin{array}{ll}
      D_i\oplus 0 & \text{if $i\in\{1,\ldots,m\}$ and $D_i\cdot e\leq 0$} \\
      D_i\oplus (-D_i\cdot e) & \text{if $i\in\{1,\ldots,m\}$ and $D_i\cdot e> 0$} \\
      0\oplus 1 & \text{If $i=m+1$}.
\end{array} 
\right. \]
The blown-up variety $\tilde X$ is the toric Deligne--Mumford stack given by the stability condition $\tilde{\omega}:=(\omega_0,-\varepsilon)$, where $\omega_0$ is in the relative interior of $W\cap \overline{C_+}=W\cap \overline{C_-}$ and $\varepsilon$ is a positive and sufficiently small real number. 


Let 
\[
Z=\cap_{i\in M_+} \{z_i=0\}\subset X
\]
be the center of the blow-up. Then the exceptional divisor $E\subset \tilde X$ is a projectivization of the normal bundle of $Z$ in $X$:
\[
\pi: E:=\mathbb P(N_{Z/X})\rightarrow Z.
\]
Then the fiber curve classes are contracted by $\pi$. Let $w_i=D_i\cdot e$ for $i\in M_+$, then the fiber of $\pi$ is $\mathbb P[w_1,\ldots, w_k]$, where $k=\on{codim} Z=\#\{M_+\}$. We would like to know what is the subvariety that is contracted under the transition.

\begin{itemize}
    \item If $\dim Z=0$, we can assume $Z$ is a point without loss of generality. Then $E=\mathbb P[w_1,\ldots, w_n]$. The intersection of the hypersurface $D_+$ and $E$ is of dimension $n-2$. The transition $\xymatrix{
\tilde D\ar@{-->}[r]& D}$ contracts a hypersurface $S$ in the weighted projective space $\mathbb P[w_1,\ldots, w_n]$ of degree $(-1+\sum_{i=1}^n w_i)$, where $n=\dim X_\pm=\#\{M_+\}$.
\item If $\dim Z>0$, let $k=\on{codim} Z=\#\{M_+\}$. Then $\dim (Z\cap D)=n-1-k$ and $\dim (E\cap \tilde D)=n-2$, the transition $\xymatrix{
\tilde D\ar@{-->}[r]& D}$ contracts a $(k-1)$-dimensional space which is the fibers $\mathbb P[w_1,\ldots, w_k]$ over $Z$.
\end{itemize}
 
 \begin{example}
 Let 
 \[
 X=\mathbb P^4 \text{ and, } \tilde X=\on{Bl}_{\on{pt}}\mathbb P^4.
 \]
 Let $Q_5\subset \mathbb P^4$ be a smooth quintic threefold and $\tilde Q_5\subset \tilde X$ be a smooth Calabi--Yau threefold defined by a regular section of the anticanonical bundle $-K_{\tilde X}$. The map $\xymatrix{
\tilde Q_5\ar@{-->}[r]& Q_5}$ contracts a cubic surface $S$ in $\mathbb P^3$. Then $Q_5$ and $\tilde Q_5$ are related by cubic extremal transition which was studied by \cite{Mi}.
 \end{example}
 
\begin{example}
Let
\[
X=\mathbb P^4 \text{ and, } \tilde X=\on{Bl}_{\mathbb P^2} \mathbb P^4.
\]
Let $Q_5\subset \mathbb P^4$ be a smooth quintic threefold and $\tilde Q_5\subset \tilde X$ be a smooth Calabi--Yau threefold defined by a regular section of the anticanonical bundle $-K_{\tilde X}$. Then $Q_5$ and $\tilde Q_5$ are related by a conifold transition. The map $\xymatrix{
\tilde Q_5\ar@{-->}[r]& Q_5}$ contracts $\mathcal O_{\mathbb P^1}(-1)\oplus \mathcal O_{\mathbb P^1}(-1)$-curves.  
\end{example}

We would like to consider the local Gromov--Witten theory of the total space of the normal bundle of $S$ or $\mathbb P[w_1,\ldots, w_k]$ in $\tilde D$. On can apply quantum Serre duality to the first case. Then, in both cases, we can consider the local Gromov--Witten theory of $\mathcal O_{\mathbb P}(-1)\oplus \mathcal O_{\mathbb P}(1-\sum_{i=1}^k w_i)$, where $\mathbb P:=\mathbb P[w_1,\ldots, w_k]$ and $k=\on{codim} Z=\#\{M_+\}$. 


Recall that the main result of \cite{MS} (see also \cite{Mi} for a cubic extremal transition) can be stated in terms of $I$-functions as follows. There is an explicit degree-preserving linear transformation $L: H^*(\tilde D)\rightarrow H^*(D)$ such that $I_D(y)$ is recovered by
\[
I_{ D}( y)=\lim_{y_{\mathrm{r}+1}\rightarrow 0} L\circ \bar I_{\tilde D}(y),
\]
where $\bar I_{\tilde D}(y)$ is obtained from $I_{\tilde D}(\tilde y)$ via analytic continuation. 


\begin{theorem}\label{thm-transition}
Let $\iota$ be the inclusion map $\iota: \mathbb P\hookrightarrow \tilde X$. Then $\iota^*I_{\tilde D}( \tilde y)|_{ \tilde y_i=0, i\in \{1,\ldots, \mathrm{r}\}}$ coincides with $I_{\mathcal O_{\mathbb P}(-1)\oplus \mathcal O_{\mathbb P}(1-\sum_{i=1}^k w_i)}(\tilde y_{\mathrm{r}+1})$. Under this identification, we have
\[
\lim_{y_{\mathrm{r}+1}\rightarrow 0} \bar I_{\mathcal O_{\mathbb P}(-1)\oplus \mathcal O_{\mathbb P}(1-\sum_{i=1}^k w_i)}( y)=0,
\]
where $\bar I_{\mathcal O_{\mathbb P}(-1)\oplus \mathcal O_{\mathbb P}(1-\sum_{i=1}^k w_i)}(y)$ is obtained from $ I_{\mathcal O_{\mathbb P}(-1)\oplus \mathcal O_{\mathbb P}(1-\sum_{i=1}^k w_i)}(\tilde y_{\mathrm{r}+1})$ after the analytic continuation in \cite{MS}.
\end{theorem}
\begin{proof}

Recall that the local $I$-function for $\mathcal O_{\tilde X}(-\tilde D)$ is
\[
I_{\mathcal O_{\tilde X}(-\tilde D)}(\tilde y,z)=ze^{\tilde t/z}\sum_{d\in\tilde{\mathbb K}}\tilde y^{d}\left(\prod_{i=0}^{m+1}\frac{\prod_{a\leq 0,\langle a\rangle=\langle \tilde D_i\cdot d\rangle}(\bar{\tilde D}_i+az)}{\prod_{a \leq \tilde D_i\cdot d,\langle a\rangle=\langle \tilde D_i\cdot d\rangle}(\bar{ \tilde D}_i+az)}\right)\frac{\prod_{a\leq 0,\langle a\rangle=\langle -\tilde D\cdot d\rangle}(-\bar{\tilde D}+az)}{\prod_{a \leq -\tilde D\cdot d,\langle a\rangle=\langle -\tilde D\cdot d\rangle}(-\bar{ \tilde D}_i+az)}\textbf{1}_{[-d]},
\]
where $\tilde{\mathbb K}=\mathbb K\oplus \mathbb R(0,-1)$, $\tilde t=\sum_{a=1}^{\mathrm{r}+1} \bar p_a \log \tilde y_a$, $ \tilde y^d= \tilde y_1^{p_1\cdot d}\cdots \tilde y_{\mathrm{r}+1}^{p_{\mathrm{r}+1}\cdot d}$ and, $[-d]$ is the equivalence class of $-d$ in $\mathbb K/ \mathbb L$.

We restrict to the $I$-function to $\tilde y_a=0$ for $1\leq a \leq \mathrm{r}$ and consider the pullback of it under the inclusion map $\iota: \mathbb P\hookrightarrow \tilde X$. Then we have the $I$-function for $\mathcal O_{\mathbb P}(-1)\oplus \mathcal O_{\mathbb P}(1-\sum_{i=1}^k w_i)$:
\begin{align*}
&I_{\mathcal O_{\mathbb P}(-1)\oplus \mathcal O_{\mathbb P}(1-\sum_{i=1}^k w_i)}( \tilde y_{\mathrm{r}+1},z)=ze^{P/z}\sum_{\substack{n\in \mathbb Q_{\geq 0}\\ \exists j: nw_j\in \mathbb Z }} \tilde y_{\mathrm{r}+1}^{n}\left(\prod_{i\in M_+}\frac{1}{\prod_{a \leq w_in,\langle a\rangle=\langle  w_in\rangle}(w_iP+az)}\right)\\
\qquad &\cdot\frac{1}{\prod_{a \leq -n,\langle a\rangle=\langle  -n\rangle}(-P+az)}\frac{1}{\prod_{a \leq 1-\sum_{i=1}^kw_in,\langle a\rangle=\langle  1-\sum_{i=1}^kw_in\rangle}((1-\sum_{i=1}^kw_i)P+az)}\textbf{1}_{[-n]},
\end{align*}
where $P\in H^2(\mathbb P)$ is the hyperplane class.

Moreover, after analytic continuation, $I_{\mathcal O_{\mathbb P}(-1)\oplus \mathcal O_{\mathbb P}(1-\sum_{i=1}^k w_i)}(\tilde y_{\mathrm{r}+1},z)$ becomes a function of $y_{\mathrm{r}+1}$, hence specializes to zero under the specialization in \cite{MS}*{Theorem 1.3} where the specialization $y_{\mathrm{r}+1}=0$ was considered.

\end{proof}

\begin{remark}
The main result of \cite{MS} is obtained via quantum Serre duality and the specialization of a variable is related to a partial compactification of the corresponding line bundle. The geometric explanation in remark \ref{rmk-partial-compactification} also provides an indication of Theorem \ref{thm-transition}. 
\end{remark}

\begin{remark}
In \cite{Mi}, the rank reduction phenomenon was partially explained as the FJRW theory for an example of cubic extremal transitions. It was shown that the regularized FJRW $I$-function satisfies a Picard--Fuchs equation which, after analytic continuation, is satisfied by the $I$-function of the cubic extremal transitions. Here, we give a general explanation of it directly in terms of Gromov--Witten theory. One may wonder how the $I$-function for local/relative Gromov--Witten theory that we considered here is related to the regularized FJRW $I$-function. We will answer this question in the next section and understand it as a result of the \emph{Landau--Ginzburg/ log Calabi--Yau correspondence}. The analytic continuation  for the example in \cite{Mi} will be explicitly worked out. 
\end{remark}

\section{Connection to FJRW theory}

Landau--Ginzburg/Calabi--Yau (LG/CY) correspondence  arises from a variation of the GIT quotient in Witten’s gauged linear sigma model (GLSM) \cite{Witten}. The LG/CY correspondence describes a relationship between sigma models based on Calabi-Yau hypersurfaces in weighted projective spaces and the Landau-Ginzburg model of the defining equation of the Calabi-Yau. The LG/CY correspondence was generalized to the Fano case and general type case by \cite{Acosta}. In the Calabi--Yau case, the $I$-functions of Gromov--Witten theory and FJRW theory are related by analytic continuation. However, the analytic continuation is impossible in the Fano case because the Picard--Fuchs operator develops an irregular singularity at the Landau--Ginzburg point and the $I$-function has radius of convergence equal to zero at the irregular singular point. In \cite{Acosta}, the author used asymptotic expansions to relate the $I$-functions. This is known as the Landau--Ginzburg/Fano (LG/Fano) correspondence.

We propose a different version of the LG/Fano correspondence by using relative Gromov--Witten invariants instead of absolute Gromov--Witten invariants of the Fano variety. We believe it is natural when we consider mirror symmetry for Fano varieties. As mentioned in Section \ref{sec:exchange-divisors}, when we consider a Landau--Ginzburg model $(X^\vee, W)$ as a mirror to a Fano variety $X$, we usually expect the generic fiber of the Landau--Ginzburg model is mirror the the smooth anticanonical divisor $D$ of the Fano variety $X$. When we consider mirror symmetry for a Fano variety $X$, it is usually natural to consider mirror symmetry for the log Calabi--Yau pair $(X,D)$ instead. Hence, relative Gromov--Witten invariants of $(X,D)$ appear naturally in mirror symmetry. Instead of considering the LG/Fano correspondence, we consider the \emph{Landau--Ginzburg/log Calabi--Yau (LG/(log CY)) correspondence}. In terms of $I$-functions, we claim that the $I$-function for the log Calabi--Yau pair $(X,D)$ can be analytically continued to the regularized $I$-function for the FJRW theory.

In Section \ref{sec:transition}, we explained that the rank reduction phenomenon is related to the local Gromov--Witten invariants of weighted projective spaces. On the other hand, \cite{Mi} explains how the rank reduction phenomenon is related to the FJRW theory in an example of cubic extremal transition. By the local-orbifold correspondence of \cite{vGGR}. \cite{TY20b} and \cite{BNTY}, it is natural to expect the corresponding relative Gromov--Witten invariants are related to the FJRW theory of the singularity. 

\subsection{Example: cubic surface} We begin with the example in \cite{Mi}. We claim that the regularized FJRW $I$-function is actually obtained by the analytic continuation of the local Gromov--Witten theory of the total space of the canonical bundle $K_S$, where $S$ is a cubic surface. Recall that the $I$-function for $K_S$ is
\begin{align}
    I_{K_S}(q)=(3H)q^{H/z}\sum_{d\geq 0}q^d\frac{\prod_{k=1}^{3d}(3H+kz)(-1)^d\prod_{k=0}^{d-1}(H+kz)}{\prod_{k=1}^d(H+kz)^4}.
\end{align}

To study analytic continuation, we write the $I$-function in terms of $\Gamma$-functions:
\begin{align*}
    I_{K_S}(q)&=3Hq^{H/z}\frac{\Gamma(1+\frac H z)^4\Gamma(1-\frac H z)}{\Gamma(1+\frac{3H}{z})}\sum_{d\geq 0}q^d\frac{\Gamma (1+\frac{3H}{z}+3d)}{\Gamma(1+\frac H z +d)^4\Gamma (1-\frac H z-d)}\\
    &=3Hq^{H/z}\frac{\Gamma(1+\frac H z)^4\Gamma(1-\frac H z)}{\Gamma(1+\frac{3H}{z})}\sum_{d\geq 0}\on{Res}_{s=d}\frac{2\pi i}{e^{2\pi i s}-1}q^s\frac{\Gamma (1+\frac{3H}{z}+3s)}{\Gamma(1+\frac H z +s)^4\Gamma (1-\frac H z-s)}.
\end{align*}

We consider the contour integral
\begin{align}\label{contour-int}
3Hq^{H/z}\frac{\Gamma(1+\frac H z)^4\Gamma(1-\frac H z)}{\Gamma(1+\frac{3H}{z})}\int_C\frac{1}{e^{2\pi i s}-1}q^s\frac{\Gamma (1+\frac{3H}{z}+3s)}{\Gamma(1+\frac H z +s)^4\Gamma (1-\frac H z-s)},
\end{align}
where the contour $C$ is chosen such that the poles at $s=l$ are on the right of $C$ and the poles at $s=-1-l$ and at 
\[
s=-\frac{H}{z}-\frac{m}{3},  m\geq 1
\]
are on the left of $C$, where $l$ is a non-negative integer.

Then (\ref{contour-int}) is the sum of residue at
\[
s=-1-l, l\geq 0 \quad \text{and} \quad s=-\frac{H}{z}-\frac{m}{3}, m\geq 1.
\]
The residue at $s=-1-l$ vanishes because they are multiples of $H^4=0$.
Note that 
\[
\on{Res}_{s=-\frac{H}{z}-\frac{m}{3}} \Gamma (1+\frac{3H}{z}+3s)=-\frac 13 \frac{(-1)^m}{\Gamma(m)}.
\]
We obtain the following analytic continuation of $I_{K_S}$:
\begin{align}
    \bar{I}_{K_S}=-H\frac{\Gamma(1+\frac H z)^4\Gamma(1-\frac H z)}{\Gamma(1+\frac{3H}{z})}\sum_{m\geq 0}(-1)^m\frac{(2\pi i)e^{\frac{2\pi im}{3}}}{e^{-2\pi i \frac{H}{z}}-e^{\frac{2\pi im}{3}}}\frac{q^{-\frac m 3}}{\Gamma(m)\Gamma(1-\frac m 3)^4\Gamma(1+\frac m 3)}.
\end{align}
Let $t^3=q^{-1}$, then we have 
\begin{align*}
     \bar{I}_{K_S}&=-H\frac{\Gamma(1+\frac H z)^4\Gamma(1-\frac H z)}{\Gamma(1+\frac{3H}{z})}\sum_{m\geq 0}(-1)^m\frac{(2\pi i)e^{\frac{2\pi im}{3}}}{e^{-2\pi i \frac{H}{z}}-e^{\frac{2\pi im}{3}}}\frac{t^{ m }}{\Gamma(m)\Gamma(1-\frac m 3)^4\Gamma(1+\frac m 3)}\\
     &=-H\frac{\Gamma(1+\frac H z)^4\Gamma(1-\frac H z)}{\Gamma(1+\frac{3H}{z})}\sum_{m\geq 0}(-1)^m\frac{(2\pi i)e^{\frac{2\pi im}{3}}}{e^{-2\pi i \frac{H}{z}}-e^{\frac{2\pi im}{3}}}\frac{t^{ m }\Gamma(\frac m 3)^4\sin^4(\frac m 3 \pi)}{\pi^4\Gamma(m)\Gamma(1+\frac m 3)}\\
     &=-H\frac{\Gamma(1+\frac H z)^4\Gamma(1-\frac H z)}{\Gamma(1+\frac{3H}{z})}\sum_{k=1,2}\frac{(2\pi i)e^{\frac{2\pi ik}{3}}}{e^{-2\pi i \frac{H}{z}}-e^{\frac{2\pi ik}{3}}}\frac{(-1)^k}{\Gamma(\frac k 3)^4\Gamma (1-\frac k 3)^4}\sum_{l\geq 0}\frac{t^{ k+3l }\Gamma(\frac {k+3l}{3})^4}{\Gamma(k+3l)\Gamma(1+\frac k 3+l)}\\
     &=-H\frac{\Gamma(1+\frac H z)^4\Gamma(1-\frac H z)}{\Gamma(1+\frac{3H}{z})}\sum_{k=1,2}(-1)^k\frac{(2\pi i)e^{\frac{2\pi ik}{3}}}{e^{-2\pi i \frac{H}{z}}-e^{\frac{2\pi ik}{3}}}\frac{(-1)^k}{\Gamma (1-\frac k 3)^4}\sum_{l\geq 0}\frac{t^{ k+3l }\Gamma(\frac {k}{3}+l)^4}{\Gamma(\frac k 3)^4\Gamma(k+3l)\Gamma(1+\frac k 3+l)}.
\end{align*}

Let $W=x_1^3+x_2^3+x_3^3+x_4^3$, we consider the FJRW theory of the pair $(W,G)$, where $G=\langle J_W\rangle$ is generated by
\[
J_W:=\left(\exp\left(\frac{2\pi i}{3}\right),\exp\left(\frac{2\pi i}{3}\right),\exp\left(\frac{2\pi i}{3}\right),\exp\left(\frac{2\pi i}{3}\right) \right)\in (\mathbb C^\times)^4.
\]
Recall that the regularized $I$-function for FJRW theory of $(W,G)$ is given in \cite{Acosta} as follows
\begin{align*}
I^{\on{reg}}_{\on{FJRW}}(t):=&I^{\on{reg}}_{\on{FJRW},1}\phi_0+I^{\on{reg}}_{\on{FJRW},2}\phi_1\\
=&\sum_{l\geq 0}\frac{t^{ 1+3l }\Gamma(\frac {1}{3}+l)^4}{\Gamma(\frac 1 3)^4(3l)!\Gamma(\frac 4 3+l)}\phi_0+\sum_{l\geq 0}\frac{t^{ 2+3l }\Gamma(\frac {2}{3}+l)^4}{\Gamma(\frac 2 3)^4(1+3l)!\Gamma(\frac 5 3+l)}\phi_1,
\end{align*}
where $\phi_0$ and $\phi_1$ are generators of the narrow state space of the FJRW theory of $(W,G)$.
Therefore, we have
\begin{align*}
     \bar{I}_{K_S}=-H\frac{\Gamma(1+\frac H z)^4\Gamma(1-\frac H z)}{\Gamma(1+\frac{3H}{z})}\sum_{k=1,2}(-1)^k\frac{(2\pi i)e^{\frac{2\pi ik}{3}}}{e^{-2\pi i \frac{H}{z}}-e^{\frac{2\pi ik}{3}}}\frac{(-1)^k}{\Gamma (1-\frac k 3)^4}I^{\on{reg}}_{\on{FJRW},k}.
\end{align*}
Hence, the local $I$-function of $K_S$ can be analytically continued to the regularized $I$-function of the FJRW theory of cubic singularity. By the local-orbifold correspondence  and the relative-orbifold correspondence \cite{ACW}, \cite{vGGR}, \cite{TY20b},  \cite{TY18}, we may identify the narrow relative quantum D-module of $(S,-K_S)$ with the narrow quantum D-module of the FJRW theory of $(W,G)$. Therefore, we conclude the following.
\begin{theorem}\label{thm-log-lg-cy-cubic}
the narrow relative quantum D-module of $(S,-K_S)$ can be identified with the narrow quantum D-module of $(W,G)$ via the analytic continuation. 
\end{theorem}

\begin{remark}
Since $(S,-K_S)$ is log Calabi--Yau, we may consider Theorem \ref{thm-log-lg-cy-cubic} as a LG/(log CY) correspondence. The meaning on the Gromov--Witten side is clear: one considers the relative Gromov--Witten theory of log Calabi--Yau pairs. On the other hand, the meaning on the FJRW side is not clear to us at this moment. It may be interesting to find out if there is a direct enumerative meaning of the regularized FJRW $I$-function in the quantum singularity theory related to the log Calabi--Yau pair $(X,D)$.
\end{remark}

\subsection{Fano hypersurfaces in weighted projective spaces} In general, we can consider a degree $d$ Fano hypersurface $X$ in a weighted projective space $\mathbb P[w_1,\ldots, w_{N}]$ along with the smooth anticanonical divisor $D$ of $X$ and study the relation between relative Gromov--Witten theory of $(X,D)$ and the FJRW theory. The set-up is the following. 

Given a Fermat polynomial
\[
W=x_1^{d/w_1}+\cdots+x_n^{d/w_N},
\]
where $\on{gcd}(w_1,\ldots,w_N)=1$ and $W$ has a unique singularity at the origin. We also assume that $w_i$ divides $d$ for all $i\in\{1,\ldots, N\}$. The polynomial $W$ is quasi-homogeneous of degree $d$:
\[
W(\lambda^{w_1}x_1,\ldots,\lambda^{w_N}x_N)=\lambda^dW(x_1,\ldots,x_N),
\]
for all $\lambda\in \mathbb C$. Let $q_i:=w_i/d$ for $i\in \{1,\ldots,N\}$ and $G=\langle J_W\rangle$, where
\[
J_W:=\left(\exp\left(2\pi i q_1\right),\ldots,\exp\left(2\pi iq_N\right)\right)\in (\mathbb C^\times)^N.
\]

Define the set 
\[
\on{Nar}:=\{k \in \{0,\ldots, d-1\}| (k+1)w_i\not\in d\mathbb Z, \text{ for all } i=1,\ldots, N\}.
\]

The regularized FJRW $I$-function is given in \cite{Acosta}:
\[
I^{\on{reg}}_{\on{FJRW}}(\tau)=\sum_{k\in \on{Nar}}\sum_{l=0}^\infty \frac{\tau^{d^\prime(l+\frac{k+1}{d})}}{(dl+k)!}\frac{(-1)^{dl+k+1}}{\Gamma(1+d^\prime\frac{dl+k+1}{d})}\prod_{i=1}^N\frac{(-1)^{\lfloor q_i(dl+k)\rfloor}\Gamma(q_i(dl+k+1))}{\Gamma(q_i+\langle q_ik\rangle)}\phi_k,
\]
where $d^\prime:=-d+\sum_{i=1}^N w_i$ and $\{\phi_k\}_{k\in \on{Nar}}$ is a basis of the narrow state space for the FJRW theory of $(W,G)$.

On the Gromov--Witten side, we consider the relative Gromov--Witten theory of a degree $d$ Fano hypersurface $X$ with its smooth anticanonical divisor. In here, we just need genus zero relative Gromov--Witten invariants with maximal contact orders. 
 We further assume that $w_i$ divides $d^\prime:=-d+\sum_{i=1}^N w_i$ for all $i\in\{1,\ldots, N\}$.

\begin{theorem}\label{thm-log-lg-cy}
The narrow relative quantum D-module of $(X,-K_X)$ can be identified with the narrow quantum D-module of $(W,G)$ via analytic continuation. 
\end{theorem}

\begin{proof}
By the local-orbifold correspondence \cite{vGGR}, \cite{TY20b} and \cite{BNTY} and the relative-orbifold correspondence\cite{ACW} and \cite{TY18}, we can consider the genus zero local Gromov--Witten theory of the total space of the canonical bundle $K_X$ of $X$. 

The $I$-function for $K_X$ is
\begin{align}
    I_{K_X}(q)=(dH)q^{H/z}\sum_{\substack{n\in \mathbb Q_{\geq 0}\\ \exists j: nw_j\in \mathbb Z }}q^n\frac{\prod_{k=1}^{nd}(dH+kz)(-1)^{nd^\prime}\prod_{k=1}^{nd^\prime}(d^\prime H+kz)}{\prod_{i=1}^N\prod_{0<k\leq nw_i, \langle k\rangle=\langle nw_i\rangle}(w_iH+kz)}\textbf{1}_{\langle -n\rangle }.
\end{align}

The $I$-function can be written in terms of $\Gamma$-functions
\begin{align*}
    I_{K_X}(q)&=dHq^{H/z}\frac{\prod_{i=1}^N \Gamma(1+\frac{w_i H}{z}-\langle -w_in\rangle)\Gamma(1-d^\prime\frac H z)}{\Gamma(1+\frac{dH}{z})}\\
    &\cdot\sum_{\substack{n\in \mathbb Q_{\geq 0}\\ \exists j: nw_j\in \mathbb Z }}q^n\frac{\Gamma (1+\frac{dH}{z}+nd)}{\Gamma (1-\frac {d^\prime H}{z}-nd^\prime)\prod_{i=1}^N\Gamma(1+\frac{w_i H}{z} +w_in)}\textbf{1}_{\langle -n\rangle }\\
    &=dHq^{H/z}\frac{\prod_{i=1}^N \Gamma(1+\frac{w_i H}{z}-\langle -w_in\rangle)\Gamma(1-d^\prime\frac H z)}{\Gamma(1+\frac{dH}{z})}\\
    &\cdot\sum_{\substack{n\in \mathbb Q_{\geq 0}\\ \exists j: nw_j\in \mathbb Z }}\on{Res}_{s=n}\frac{2\pi i}{e^{2\pi i s}-1}q^s\frac{\Gamma (1+\frac{dH}{z}+sd)}{\Gamma (1-\frac {d^\prime H}{z}-sd^\prime)\prod_{i=1}^N\Gamma(1+\frac{w_i H}{z} +w_is)}\textbf{1}_{\langle -n\rangle }.
\end{align*}

Following \cite{CIR}, we consider the $H$-functions. The regularized $H$-function for the FJRW theory is
\begin{align*}
H^{\on{reg}}_{\on{FJRW}}(\tau)&=\sum_{k\in \{0,\ldots,d-1\}}\sum_{l=0}^\infty \frac{\tau^{d^\prime(l+\frac{k+1}{d})}}{(dl+k)!}\frac{(-1)^{dl+k+1}}{\Gamma(1+d^\prime\frac{dl+k+1}{d})}\prod_{j=1}^N (-1)^{\lfloor q_j(dl+k)\rfloor}\Gamma(q_j(dl+k+1))\phi_k\\
&=\sum_{k\in \{0,\ldots,d-1\}}H^{\on{reg}}_{\on{FJRW},k}\phi_k.
\end{align*}
The $H$-function for $I_{K_X}$ is 
\begin{align*}
H_{K_X}(q)&=dHe^{\frac{H}{2\pi i}\log q}\sum_{\substack{n\in \mathbb Q_{\geq 0}\\ \exists j: nw_j\in \mathbb Z }}q^n\frac{\Gamma (1+\frac{dH}{2\pi i}+nd)}{\Gamma (1-\frac {d^\prime H}{2\pi i}-nd^\prime)\prod_{i=1}^N\Gamma(1+\frac{w_i H}{2\pi i} +w_in)}\textbf{1}_{\langle -n\rangle }\\
&=dHe^{\frac{H}{2\pi i}\log q}\sum_{\substack{n\in \mathbb Q_{\geq 0}\\ \exists j: nw_j\in \mathbb Z }}\on{Res}_{s=n}\frac{2\pi i}{e^{2\pi i s}-1}q^s\frac{\Gamma (1+\frac{dH}{2\pi i}+sd)}{\Gamma (1-\frac {d^\prime H}{2\pi i}-sd^\prime)\prod_{i=1}^N\Gamma(1+\frac{w_i H}{2\pi i} +w_is)}\textbf{1}_{\langle -n\rangle }\\
&=\sum_{f}H_{K_X,f}\textbf{1}_{f},
\end{align*}
where $f$ runs over the set
\[
\{0\leq f< 1| fw_j\in \mathbb Z \text{ for some } 1\leq j \leq N\},
\] 
$\bar{f}=\langle 1-f\rangle$ and
\[
H_{K_X,f}=dHe^{\frac{H}{2\pi i}\log q}\sum_{n\in \mathbb Z_{\geq 0}}\on{Res}_{s=n}\frac{2\pi i}{e^{2\pi i (s+\bar{f})}-1}q^{s+\bar{f}}\frac{\Gamma (1+\frac{dH}{2\pi i}+sd+\bar{f}d)}{\Gamma (1-\frac {d^\prime H}{2\pi i}-sd^\prime+\bar{f}d^\prime)\prod_{i=1}^N\Gamma(1+\frac{w_i H}{2\pi i} +w_is+\bar{f}w_i)}.
\]

We consider the contour integral 
\[
\int_C \frac{1}{e^{2\pi i (s+\bar{f})}-1}q^{s+\bar{f}}\frac{\Gamma (1+\frac{dH}{2\pi i}+sd+\bar{f}d)}{\Gamma (1-\frac {d^\prime H}{2\pi i}-sd^\prime+\bar{f}d^\prime)\prod_{i=1}^N\Gamma(1+\frac{w_i H}{2\pi i} +w_is+\bar{f}w_i)},
\]
where the contour $C$ is chosen such that the poles at $s=l$ are on the right of $C$ and the poles at $s=-1-l$ and at 
\[
s=-\frac{H}{2\pi i}-\frac{m}{d}-\bar{f},  m\geq 1
\]
are on the left of $C$, where $l$ is a non-negative integer.
Then the contour integral is the sum of residue at
\[
s=-1-l, l\geq 0 \text{ and } s=-\frac{H}{2\pi i}-\frac{m}{d}-\bar{f}, m\geq 1.
\]
The residue at $s=-1-l$ vanishes.

We have
\[
\on{Res}_{s=-\frac{H}{2\pi i}-\frac{m}{d}} \Gamma (1+\frac{dH}{2\pi i}+ds)=-\frac 1d \frac{(-1)^m}{\Gamma(m)}.
\]
Therefore, we obtain an analytic continuation of $H_{K_X,f}$:
\begin{align*}
dHe^{\frac{H}{2\pi i}}\sum_{m\geq 1}(-1)^m\frac{(2\pi i)e^{\frac{2\pi im}{d}}}{e^{-H}-e^{\frac{2\pi im}{d}}}\frac{q^{-\frac m d}}{\Gamma(m)\Gamma(1+\frac{d^\prime}{d}m )\prod_{i=1}^N\Gamma(1-q_im)}.
\end{align*}
Let $\tau^d=q^{-1}$ and $m=k+dl+1$, where $l\in \mathbb Z_{\geq 0}$. Then we have 
\begin{align*}
&dHe^{\frac{H}{2\pi i}}\sum_{k\in\{0,\ldots d-1\}}\sum_{l=0}^\infty(-1)^{k+dl+1}\frac{(2\pi i)e^{\frac{2\pi i(k+1)}{d}}}{e^{-H}-e^{\frac{2\pi i(k+1)}{d}}}\frac{q^{-\frac {k+dl+1} d}\prod_{i=1}^N\Gamma(q_i(k+dl+1))\sin(\pi q_i(k+dl+1))}{\pi^N(k+dl)!\Gamma(1+\frac{d^\prime}{d}(k+dl+1) )}\\
=&dHe^{\frac{H}{2\pi i}}\sum_{k\in\{0,\ldots d-1\}}(-1)^{k+1}\frac{(2\pi i)e^{\frac{2\pi i(k+1)}{d}}}{e^{-H}-e^{\frac{2\pi i(k+1)}{d}}}\frac{1}{\Gamma(1-q_i(k+1))\Gamma(q_i(k+1))}\sum_{l=0}^\infty\frac{q^{-\frac {k+dl+1} d}\prod_{i=1}^N\Gamma(q_i(k+dl+1))}{(k+dl)!\Gamma(1+\frac{d^\prime}{d}(k+dl+1) )}\\
=&dHe^{\frac{H}{2\pi i}}\sum_{k\in\{0,\ldots d-1\}}(-1)^{k+1}\frac{(2\pi i)e^{\frac{2\pi i(k+1)}{d}}}{e^{-H}-e^{\frac{2\pi i(k+1)}{d}}}\frac{1}{\Gamma(1-q_i(k+1))\Gamma(q_i(k+1))}H^{\on{reg}}_{\on{FJRW},k}.
\end{align*}
Therefore, we obtained an analytic continuation between the $I$-function for $K_X$ and the regularized $I$-function for the FJRW theory of $(W,G)$.
\end{proof}

\bibliographystyle{amsxport}
\bibliography{main}

\end{document}